\documentclass{amsart}

\usepackage[a4paper,hmargin=3.5cm,vmargin=4cm]{geometry}
\usepackage{amsfonts,amssymb,amscd,amstext}
\usepackage{graphicx}
\usepackage[dvips]{epsfig}
\usepackage{mathpazo}
\usepackage{enumerate}

\newcommand{\dist}{\operatorname{dist}}

\def\div{\mathfrak{Div}}

\def\r{\mathbb{R}}
\def\n{\mathbb{N}}
\def\c{\mathbb{C}}

\def\d{\mathbb{D}}

\def\z{\mathbb{Z}}

\def\k{\mathbb{K}}

\def\Ncal{\mathcal{N}}
\def\Mcal{\mathcal{M}}
\def\Ccal{\mathcal{C}}
\def\Fcal{\mathcal{F}}
\def\Ical{\mathcal{I}}

\def\hgot{\mathfrak{h}}
\def\mgot{\mathfrak{m}}
\def\ggot{\mathfrak{g}}

\newtheorem{theorem}{Theorem}[section]

\newtheorem{assertion}[theorem]{Claim}
\newtheorem{lemma}[theorem]{Lemma}
\newtheorem{corollary}[theorem]{Corollary}

\theoremstyle{definition}
\newtheorem{remark}[theorem]{Remark}
\newtheorem{definition}[theorem]{Definition}

\numberwithin{equation}{section}
\numberwithin{figure}{section}

\begin{document}

\title[Minimal surfaces in $\r^3$  properly projecting into $\r^2$]
{Minimal surfaces in $\r^3$ properly projecting into $\r^2$}

\author[A.~Alarc\'{o}n]{Antonio Alarc\'{o}n}
\address{Departamento de Geometr\'{\i}a y Topolog\'{\i}a \\
Universidad de Granada \\ E-18071 Granada \\ Spain}
\email{alarcon@ugr.es}

\author[F.J.~L\'{o}pez]{Francisco J. L\'{o}pez}
\address{Departamento de Geometr\'{\i}a y Topolog\'{\i}a \\
Universidad de Granada \\ E-18071 Granada \\ Spain}
\email{fjlopez@ugr.es}


\thanks{Research of both authors is partially
supported by MCYT-FEDER research project MTM2007-61775 and Junta
de Andaluc\'{i}a Grant P09-FQM-5088} 

\begin{abstract}
For all open Riemann surface $\mathcal{ N}$ and  real number $\theta \in (0,\pi/2),$  we construct a conformal minimal immersion $X=(X_1,X_2,X_3):\mathcal{ N} \to \r^3$ such that $X_3+\tan(\theta) |X_1|:\mathcal{ N} \to \r$ is positive and  proper. Furthermore, $X$ can be chosen with arbitrarily prescribed flux map.

Moreover, we produce properly immersed hyperbolic minimal surfaces with non empty boundary in $\r^3$ lying above a negative sublinear graph.
\end{abstract}

\maketitle

\thispagestyle{empty}

\section{Introduction}\label{sec:intro}
The conformal structure of a complete minimal surface plays a fundamental role in its global properties. It is then  important to determine the conformal type of a given minimal surface. An  open Riemann surface is said to be {\em hyperbolic} if an only if it carries a negative non-constant subharmonic function. Otherwise, it is said to be {\em parabolic}. Compact Riemann surfaces with empty boundary are said to be {\em elliptic}.

Complete minimal surfaces with finite total curvature or complete embedded minimal surfaces with finite topology in $\r^3$ are properly immersed and have parabolic conformal type (for further  information, see \cite{osserman, jorge-meeks, c-m1, m-p-r,m-p}). On the other hand, there exist properly immersed hyperbolic minimal surfaces in $\r^3$ with arbitrary non-compact topology (see \cite{mo} for a pioneering work, and \cite{f-m-m,AL} and references therein for a good setting).

It is then interesting to elucidate how properness and completeness influence the conformal geometry of minimal surfaces. In \cite{lop1}  it is shown that any open Riemann surface admits a conformal complete minimal immersion in $\r^3,$ even with arbitrarily prescribed flux map. In this paper we extend this result to the family of proper minimal immersions, proving considerably more (see Theorem \ref{th:fun}):

\begin{quote}
{\bf Theorem I.} {\em For all open Riemann surface $\mathcal{ N}$, group morphism $p:\mathcal{ H}_1(\mathcal{ N},\z) \to \r^3$  and real number $\theta \in (0,\frac{\pi}{2}),$ there exists a conformal minimal immersion $X=(X_1,X_2,X_3):\mathcal{ N} \to \r^3$ satisfying that:
\begin{itemize}
\item $X_3+\tan(\theta) |X_1|:\mathcal{ N} \to \r$ is positive and  proper, and
\item $\int_\gamma \partial X=i p(\gamma)$ for all $\gamma \in \mathcal{ H}_1(\mathcal{ N},\z),$ where $\partial$ is the complex differential operator.
\end{itemize}}
\end{quote}

The strength of the theorem lies in the case $\theta\approx 0.$ As a matter of fact, if the theorem holds for some $\theta_0\in (0,\pi/2)$ then it is trivially valid for  any  $\theta\in [\theta_0,\pi/2).$ Furthermore, the result is sharp in the  sense that the angle $\theta$ cannot be zero. Indeed, by the Strong Half Space Theorem \cite{h-m} properly immersed minimal surfaces in a half space are planes.  Contrariwise, Theorem I shows that any wedge of angle greater than $\pi$ in $\r^3$ contains minimal surfaces properly immersed in $\r^3,$ even of   hyperbolic type. In particular, neither open wedges nor closed wedges of angle greater than $\pi$ are universal regions for surfaces (see \cite{mp1} for a good setting). Other Picard conditions for properly immersed minimal surfaces in $\r^3$ guaranteeing parabolicity can be found in
\cite{lop3}.

From Theorem I follow some remarkable results concerning not only
minimal surfaces. We are going to mention three of them related to
proper harmonic maps into $\c,$ proper holomorphic null curves in
$\c^3$ and maximal surfaces in the Lorentz-Minkowski space
$\r^3_1.$

Schoen and Yau conjectured that there are no proper harmonic maps
from $\d$ to $\c$ with flat metrics, and connected this question
with the existence of hyperbolic minimal surfaces in $\r^3$
properly projecting into $\r^2$ \cite[p.$\,18$]{s-y}. A
counterexample to this conjecture follows from the results in
\cite{DF}, which imply the existence of proper harmonic maps from
any finite bordered Riemann surface into $\r^2.$ It
remains open whether or not a hyperbolic minimal surface in $\r^3$
can be properly projected into $\r^2.$  The following direct corollary of Theorem I provides a full answer to Schoen and Yau's questions:
\begin{quote}
{\bf Corollary.} {\em Any open Riemann surface $\mathcal{ N}$ admits a conformal minimal immersion $X=(X_1,X_2,X_3):\mathcal{ N} \to \r^3$ such that $(X_1,X_3):\mathcal{N}\to\r^2$ is a proper (harmonic) map.}
\end{quote}

It is well known that any open Riemann surface properly
holomorphically embeds in $\c^3$ and immerses in $\c^2$ \cite{bis,
nar, rem}. Moreover, there are proper null immersions in $\c^3$ of
the unit disc \cite{mo}, and of any open parabolic Riemann surface
of finite topology \cite{pirola,lop1}. Theorem I also shows that
any open Riemann surface admits a proper {\em null} immersion in
$\c^3,$ and a holomorphic immersion in $\c^2$ properly projecting
into $\r^2.$ Indeed, choosing $p=0$ in Theorem I and labeling
$X^*=(X_1^*,X_2^*,X_3^*)$ as the conjugate minimal immersion of
$X,$ the map $X+i X^*=(F_1,F_2,F_3):\mathcal{ N} \to \c^3$ is a
proper holomorphic null immersion, and  $(F_1,F_3):\mathcal{ N}
\to \c^2$ is a holomorphic immersion which properly projects into
$\r^2.$

Finally, from Theorem I follows the existence of proper Lorentzian null holomorphic immersions in $\c^3$ (see \cite{uy}) and proper conformal
maximal immersions in the Lorentz-Minkowski space, with singularities and arbitrary conformal structure. See \cite{alar} for the hyperbolic
simply connected case.

The last part of the paper is devoted to properly immersed minimal surfaces in $\r^3$ with  non-empty boundary. A Riemann surface $M$ with {\em non-empty boundary} is said to be {\em parabolic} if bounded harmonic functions on $M$ are determined by their boundary values, or equivalently, if the harmonic  measure of $M$ with respect to a point $P\in M-\partial(M)$ is full on $\partial(M).$ Otherwise, the surface is said to be {\em hyperbolic} (see \cite{ahlfors, P} for a good setting).  For instance, $\overline{\d}-\{1\}$ is parabolic whereas $\overline{\d}_+:=\overline{\d} \cap \{z \in \c\,|\, \mbox{Im}(z)>0\}$ is hyperbolic.  Properly immersed minimal surfaces with non-empty boundary lying in a half space of $\r^3$ are parabolic \cite{c-k-m-r}, and the same result holds for proper minimal graphs in $\r^3$ \cite{neel}. It is also known that any properly immersed minimal surface in $\r^3$ with non-empty boundary lying over a negative sublinear graph in $\r^3$  and whose Gaussian image is contained in a hyperbolic domain of the Riemann sphere is parabolic \cite{l-p}. We prove the following complementary result (see Theorem \ref{th:sub}), which also shows that the condition about the size of the Gauss map in  \cite{l-p} plays an important role:
\begin{quote}
{\bf Theorem II.} {\em There exists a conformal minimal immersion  $X=(X_1,X_2,X_3): \overline{\d}_+ \to \r^3$ such that $(X_1,X_3): \overline{\d}_+ \to \r^2$ is proper  and $\lim_{n \to \infty} \min \{\frac{X_3(p_n)}{|X_1(p_n)|+1},0\}=0$ for all divergent sequence $\{p_n\}_{n \in \n}$ in $ \overline{\d}_+.$}
\end{quote}

Theorem II contributes to the understanding of Meeks' conjecture about parabolicity of minimal surfaces with boundary. This conjecture asserts that any properly immersed minimal surface lying above a negative half catenoid is parabolic.

The techniques developed in this paper may be applied to a wide
range of problems on minimal surface theory. In the papers
\cite{AFL, AFL2} complete minimal surfaces in $\r^N$ with prescribed coordinate functions are
constructed, and in \cite{AL} some Calabi-Yau type conjectures are treated. Our tools come from deep results on approximation
theory by meromorphic functions \cite{sche1, sche2, roy}. The most
useful one is the Approximation Lemma in Section
\ref{sec:aproxi}, where  accurate use of Runge-Mergelyan approximation
theorems and classical theory of Riemann surfaces \cite{ahlfors,
farkas} is made. In this way, we can refine the classical
construction methods of complete minimal surfaces (see, among
others, \cite{jorge-xavier, nadi, lop-mar-mo} for a good
setting).

The paper is laid out as follows. In Section \ref{sec:riemann} we
introduce the necessary background on Riemann surfaces and the required notations for a well understanding of the subsequent sections.  Section \ref{sec:wei} is devoted to
some preliminaries on minimal surfaces in $\r^3.$ 
In Section \ref{sec:aproxi} we state and prove the Approximation Lemma. Finally,
Theorems I and II are proved in Sections \ref{sec:fun} and
\ref{sec:sub}, respectively.

\section{Background on Riemann Surfaces} \label{sec:riemann}
Given a compact topological space $K$ and $f=(f_j)_{j=1,\ldots,n}:K\to\mathbb{K}^n,$ $\mathbb{K}=\r,$ $\c,$  we denote by $$\|f\|_{0,K}:= \max_{K} \big\{ \big(\sum_{j=1}^n |f_j|^2\big)^{1/2}\big\}$$ the maximum norm of $f$ on $K.$ The corresponding space of continuous functions on $K$ will be endowed with the $\Ccal^0$ topology associated to $\|\cdot\|_{0,K}.$

Given a topological surface $N,$  $\partial(N)$ will denote the one dimensional topological manifold determined by the boundary points of $N.$ Given $A \subset N,$ call by $A^\circ$ and $\overline{A}$  the interior and the closure  of $A$  in $N,$ respectively. Open connected subsets of $N-\partial(N)$ will be called {\em domains}, and those proper connected topological subspaces of $N$ being surfaces with boundary are said to be  {\em regions}.

A Riemann surface $M$ is said to be {\em open} if it is non-compact and $\partial(M)=\emptyset.$ As usual, $\overline{\c}=\c \cup \{\infty\}$ will denote the Riemann sphere. We denote  $\partial$ as the global complex operator given by $\partial|_U=\frac{\partial}{\partial z} dz$ for any conformal chart $(U,z)$ on $M.$

\begin{remark}\label{rem:inicio}
Throughout this paper $\mathcal{ N}$ and $\sigma_{\mathcal{N}}^2$ will denote a fixed but arbitrary open Riemann surface and conformal Riemannian metric on it.
\end{remark}

In the following, we introduce the necessary  notations for a well understanding of Sections \ref{sec:wei} and \ref{sec:aproxi}.

For any $A \subset \mathcal{N},$ we denote by $\div(A)$  the free commutative group of divisors of $A$ with multiplicative notation. If $D=\prod_{i=1}^n Q_i^{n_i} \in \div(S),$ where $n_i \in \z-\{0\}$ for all $i,$ the set $\{Q_1,\ldots,Q_n\}$ is said to be the support of $D,$ written $\mbox{supp}(D).$
A divisor $D \in \div(A)$ is said to be {\em integral} if $D=\prod_{i=1}^n Q_i^{n_i}$ and $n_i\geq 0$ for all $i.$ Given $D_1,$ $D_2 \in \div(A),$ $D_1 \geq D_2$ if and only if $D_1 D_2^{-1}$ is  integral.

Given an open subset  $W\subset \mathcal{N},$ we write ${\mathcal{ F}_\hgot}(W)$ and ${\mathcal{ F}_\mgot}(W)$ for the spaces of holomorphic and meromorphic functions on $W,$ respectively. Likewise, $\Omega_\hgot(W)$ and $\Omega_\mgot(W)$ will denote the spaces of holomorphic and meromorphic 1-forms on $W,$ respectively.

Let $S$ be a {\em compact subset} of $\mathcal{ N}.$ By definition, a connected component $V$ of $\mathcal{N}-S$ is said to be {\em bounded}  if $\overline{V}$ is compact. $S$ is said to be {\em Runge}  if $\mathcal{N}-S$ has no bounded components. Recall that a compact Jordan arc in $\mathcal{ N}$ is said to be analytical (smooth, continuous,...) if it is contained in an open analytical (smooth, continuous,...) Jordan arc in $\mathcal{ N}.$
\begin{definition}
A (possibly non-connected) compact subset $S\subset \mathcal{N}$ is said to be admissible  if and only if (see Figure \ref{fig:admi}):
\begin{enumerate}[(a)]
\item $S$ is Runge,
\item $M_S:=\overline{S^\circ}$ is non-empty   and consists of a finite collection of pairwise disjoint compact regions in $W$ with   $\mathcal{ C}^0$ boundary,
\item $C_S:=\overline{S-M_S}$ consists of a finite collection of pairwise disjoint analytical Jordan arcs, and
\item any component $\alpha$ of $C_S$  with an endpoint  $P\in M_S$ admits an analytical extension $\beta$ in $\mathcal{N}$ such that the unique component of $\beta-\alpha$ with endpoint $P$ lies in $M_S.$
\end{enumerate}
\end{definition}
\begin{figure}[ht]
    \begin{center}
    \scalebox{0.30}{\includegraphics{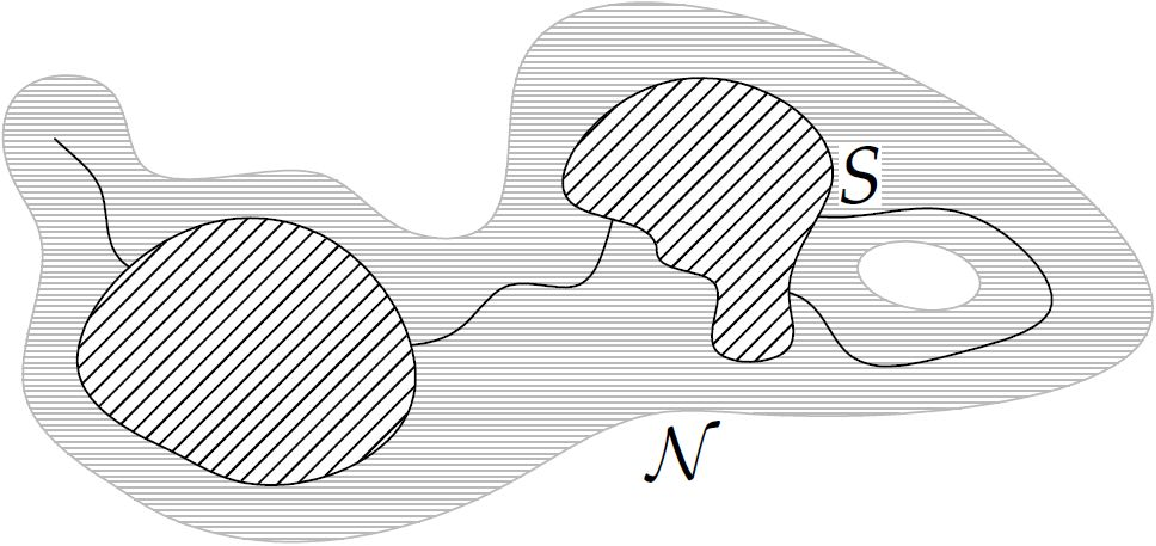}}
        \end{center}
        \vspace{-0.5cm}
\caption{An admissible set $S.$}\label{fig:admi}
\end{figure}
A compact subset $S\subset \mathcal{N}$ satisfying (b), (c) and (d)  is Runge (hence admissible) if and only if $i_*:\mathcal{H}_1(S,\z) \to \mathcal{H}_1(\mathcal{N},\z)$ is a monomorphism, where $\mathcal{H}_1(\cdot,\z)$ means first homology group, $i:S \to \mathcal{N}$  is the inclusion map  and $i_*$ is the induced group morphism. Elementary topological arguments give that $\mathcal{H}_1(S,\z)$  is finitely generated and  $\chi(M_S)\geq \chi(S) \geq \chi(M_S)-k$ for any admissible $S,$ where $\chi(\cdot)$ means Euler characteristic and $k$ is the number of Jordan arcs in $C_S.$ In particular, $\chi(S)$ is finite.

Notice that if $S\subset \mathcal{N}$ is a compact Runge subset  consisting of a finite collection of pairwise disjoint compact regions with $\mathcal{ C}^0$ boundary, then $S$  is admissible. For most of the admissible subsets $S$ we will deal with in this paper, $M_S$ will have smooth (or even analytical) boundary and the arcs in $C_S$ will meet transversally $\partial (M_S).$

In the sequel, $S$ will denote an admissible set. 
\begin{definition} We denote by
\begin{itemize}
\item ${\mathcal{ F}_\hgot}(S)$ the space of continuous functions $f:S \to{\c}$  being holomorphic on an open neighborhood  of $M_S$ in $\mathcal{ N},$ and
\item $\mathcal{ F}_\mgot(S)$ the space of continuous functions $f:S \to \overline{\c}$ being meromorphic on an open neighborhood  of $M_S$ in $\mathcal{ N}$ and satisfying that  $f^{-1}(\infty)\subset S^\circ=M_S-\partial(M_S).$ 
\end{itemize}
\end{definition}
As usual, a 1-form $\theta$ on $S$ is said to be of type $(1,0)$ if for any conformal chart $(U,z)$ in $ \mathcal{ N},$ $\theta|_{U \cap S}=h(z) dz$ for some function $h:U \cap S \to \overline{\c}.$
Finite sequences $\Theta=(\theta_1,\ldots,\theta_n),$ where $\theta_j$ is a $(1,0)$-type 1-form for all $j,$ are said to be $n$-dimensional vectorial $(1,0)$-forms on $S.$   The space of continuous $n$-dimensional (1,0)-forms on $S$ will be endowed with the $\Ccal^0$ topology induced by the norm 
\begin{equation} \label{eq:norma0}
\|\Theta\|_{0,S}:=\|\frac{\Theta}{\sigma_\Ncal}\|_{0,S}=\max_{S} \big\{ \big(\sum_{j=1}^n |\frac{\theta_j}{\sigma_\Ncal}|^2\big)^{1/2}\big\} .
\end{equation}

Fix any arbitrary  meromorphic 1-form $\vartheta_S$ on $\Ncal$ with neither zeros nor poles on $S$ (the existence of a such $\vartheta_S$ is well known, it follows from Riemann-Roch theorem on open Riemann surfaces). Notice that the following notions will not depend on the chosen $\vartheta_S.$
 
\begin{definition} We denote by
\begin{itemize}
\item $\Omega_\hgot(S)$ the space of 1-forms $\theta$ of type $(1,0)$ on $S$ such that $\theta/\vartheta_S \in \mathcal{ F}_\hgot(S),$ and 
\item $\Omega_\mgot(S)$ the space of 1-forms $\theta$ of type $(1,0)$ on $S$ such that $\theta/\vartheta_S\in \mathcal{ F}_\mgot(S).$  
\end{itemize}
\end{definition}
The inclusions $\mathcal{F}_\hgot(S)\subset  \mathcal{F}_\mgot(S)$ and $\Omega_\hgot(S)\subset  \Omega_\mgot(S)$ are trivial.

For any $f \in \mathcal{ F}_\mgot(S)$ we denote by $(f)_0$ and $(f)_\infty$ its associated integral divisors of zeroes and poles in $S,$ respectively, and label $(f)=\frac{(f)_0}{(f)_\infty}$ as the divisor associated to $f$ on $S.$ Obviously, $\mbox{supp}((f)_\infty)=f^{-1}(\infty)$ and $\mbox{supp}((f)_0)=f^{-1}(0).$ Likewise we define $(\theta)_0$, $(\theta)_\infty$  for any  $\theta \in \Omega_\mgot(S)$ and call $(\theta)=\frac{(\theta)_0}{(\theta)_\infty}$ as the  divisor of $\theta$ on $S.$ 

\begin{definition}\label{def:omega-t} Let $W$ be an open subset of $\mathcal{N}$ containing $S.$ We shall say that
\begin{itemize}
\item  a function $f \in \mathcal{ F}_\hgot(S)$ can be approximated in the $\Ccal^0$ topology on $S$ by functions in ${\mathcal{ F}_\hgot}(W)$ if
there exists $\{f_n\}_{n \in \n} \subset  {\mathcal{ F}_\hgot}(W)$ such that $\{\|f_n|_S-f\|_{0,S}\}_{n \in \n} \to 0,$
\item a function $f \in \mathcal{ F}_\mgot(S)$ can be approximated in the $\Ccal^0$ topology on $S$ by functions in $\mathcal{F}_\mgot(W)$ if
there exists $\{f_n\}_{n \in \n} \subset  {\mathcal{ F}_\mgot}(W)$ such that $f_n|_S-f \in {\mathcal{ F}_\hgot}(S)$ for all $n$ and $\{\|f_n|_S-f\|_{0,S}\}_{n \in \n} \to 0$ (in particular, $(f_n)_\infty=(f)_\infty$ on $S^\circ$ for all $n$),
\item  a 1-form $\theta \in \Omega_\hgot(S)$ can be approximated in the $\Ccal^0$ topology on $S$ by 1-forms in $\Omega_\hgot(W)$ if
there exists  $\{\theta_n\}_{n \in \n} \subset \Omega_\hgot(W)$ such that $\{\|\theta_n|_S-\theta\|_{0,S}\}_{n \in \n} \to 0,$ and 
\item a 1-form $\theta \in \Omega_\mgot(S)$ can be approximated in the $\Ccal^0$ topology on $S$ by 1-forms in $\Omega_\mgot(W)$ if
there exists  $\{\theta_n\}_{n \in \n} \subset \Omega_\mgot(W)$ such that $\theta_n|_S-\theta \in \Omega_\hgot(S)$ for all $n$ and $\{\|\theta_n|_S-\theta\|_{0,S}\}_{n \in \n} \to 0$  (in particular $(\theta_n)_\infty=(\theta)_\infty$  on $S^\circ$ for all $n$).      
\end{itemize}
\end{definition}
The notion of approximation in the $\Ccal^0$ topology of vectorial functions in $\Fcal_\hgot(S)^n$ (respectively,  1-forms in $\Omega_\hgot(S)^n$) by vectorial functions in $\Fcal_\hgot(W)^n$ (respectively,  1-forms in $\Omega_\hgot(W)^n$) is set in a similar way. Likewise for the spaces $\Fcal_\mgot(S)^n$ and $\Omega_\mgot(S)^n.$

The following definition deals with the notion of smoothness of functions and 1-forms on admissible subsets. 

\begin{definition} \label{def:smooth}
Let $S$ be a compact admissible subset in $\mathcal{N}.$ 
\begin{itemize}
\item A function $f:S\to\k^n,$ $\k=\r,$ $\c,$ or $\overline{\c},$ $n\in\n,$ is said to be smooth if $f|_{M_S}$  admits a smooth extension $f_0$ to an open domain $V$ in $\mathcal{N}$ containing $M_S,$ and for any component $\alpha$ of $C_S$ and any open analytical Jordan arc $\beta$ in $W$ containing $\alpha,$  $f|_\alpha$ admits a smooth extension $f_\beta$ to $\beta$ satisfying that $f_\beta|_{V \cap \beta}=f_0|_{V \cap \beta}.$ 
\item A vectorial 1-form $\Theta\in \Omega_\mgot(S)^n$ is said to be smooth if $\Theta/\vartheta_S:S \to \overline{\c}^n$ is smooth.
\end{itemize}
\end{definition}

\begin{definition}
Given a smooth $f\in\mathcal{ F}_\mgot(S),$ we set $df$ as the 1-form of type (1,0)  given by 
$$df|_{M_S}=d (f|_{M_S})\quad \text{and}\quad df|_{\alpha \cap U}=(f \circ \alpha)'(x)dz|_{\alpha \cap U}$$ for any component $\alpha$ of $C_S,$ 
where $(U,z=x+i y)$ is any conformal chart on $\mathcal{N}$ satisfying that $z(\alpha \cap U)\subset \r$ (the existence of such a conformal chart is guaranteed by the analyticity of $\alpha$). 
\end{definition}

It is clear that $df$ is well defined, belongs to $\Omega_\mgot(S)$ (to $\Omega_\hgot(S)$ if $f\in\mathcal{F}_\hgot(S)$) and is smooth. Furthermore, $df|_\alpha(t)= (f\circ\alpha)'(t) dt$ for any component $\alpha$ of $C_S,$ where $t$ is any smooth parameter along $\alpha.$ 

A smooth 1-form $\theta \in \Omega_\mgot(S)$ is said to be {\em exact} if $\theta=df$ for some smooth $f \in \mathcal{ F}_\mgot (S),$ or equivalently if $\int_\gamma \theta=0$ for all $\gamma \in \mathcal{ H}_1(S,\z).$

\section{Weierstrass Representation  and Flux Map of Minimal Surfaces} \label{sec:wei}
Let $R$ be an open Riemann surface and let $X=(X_1,X_2,X_3):R \to \r^3$ be a conformal minimal immersion. Denote by $\phi_j=\partial X_j,$ $j=1,2,3,$ and $\Phi=\partial X\equiv (\phi_j)_{j=1,2,3}.$  The 1-forms  $\phi_k$ are holomorphic,  have no real periods and satisfy that  $\sum_{k=1}^3 \phi_k^2=0.$ Furthermore,  the intrinsic metric in $R$ is given by
$ds^2=\sum_{k=1}^3 |\phi_k|^2,$ hence $\phi_k,$ $k=1,2,3,$ have no common zeroes.

Conversely, any vectorial holomorphic 1-form $\Phi=(\phi_1,\phi_2,\phi_3)$ on $R$ without real periods and satisfying that $\sum_{k=1}^3 \phi_k^2=0$  and $\sum_{k=1}^3 |\phi_k(P)|^2\neq 0$ for all $P \in R,$ determines a conformal minimal immersion $X:R \to \r^3$ by the expression:
$$X=\mbox{Re} \int \Phi.$$ By definition, the triple $\Phi$ is said to be the Weierstrass representation of $X.$ The meromorphic function $g=\frac{\phi_3}{\phi_1-i \phi_2}$ corresponds to the Gauss map of $X$ up to the stereographic projection and $$\Phi=\big(\frac{1}{2}(1/g-g),\frac{i}{2}(1/g+g),1\big)\phi_3,$$ see \cite{osserman}.

We need the following
\begin{definition}
For any subset $A\subset \Ncal,$  we denote by $\mathcal{ M}(A)$ the space of conformal minimal immersions of  open domains  $W\subset \Ncal$ containing $A$ into $\r^3.$ 
\end{definition}

Let $S \subset \mathcal{ N}$ be a compact  admissible subset. 
\begin{definition} \label{def:conor}
Given $X\in \mathcal{ M}(S)$  and an arclength parameterized curve  $\gamma(s)$ in $S,$ the {\em conormal vector field}  of $X$ along $\gamma$ is the unique unitary tangent vector field  $\mu$ of $X$ along $\gamma$ such that $\{d X(\gamma'(s)),\mu(s)\}$ is a positive basis for all $s.$ If in addition $\gamma$ is closed, the {\em flux} $p_X(\gamma)$ of $X$ along $\gamma$ is given by $\int_\gamma \mu(s) ds.$ 
\end{definition}
It is easy to check that $p_X(\gamma)=\mbox{Im}\int_{\gamma} \partial X$ and that the flux map $p_X:\mathcal{ H}_1(M,\z)\to \r^3$ is a group morphism.

\begin{definition}
A smooth map $X:S \to \r^3$ (see Definition \ref{def:smooth}) is said to be a {\em generalized minimal immersion} if $X|_{M_S} \in \mathcal{ M}(M_S)$ and $X|_{C_S}$ is regular, that is to say, if $X|_\alpha$ is a regular curve for all $\alpha \subset C_S.$
We denote by $\mathcal{M}_\ggot(S)$ the space of generalized minimal immersions of $S$ into $\r^3.$
\end{definition}
It is clear that $Y|_{S}\in \mathcal{M}_\ggot(S)$  for all $Y \in \mathcal{ M}(S).$ 

Consider $X \in \mathcal{ M}_\ggot(S)$  and let $\varpi:C_S \to \r^3$ be a {\em smooth normal field} along $C_S$ respect to $X.$ This simply means that for any (analytical) arclength parameterized $\alpha (s) \subset C_S,$   $\varpi(\alpha(s))$ is smooth, unitary and orthogonal to $(X|_\alpha)'(s),$    $\varpi$ extends smoothly to any open analytical arc $\beta$ in $W$ containing $\alpha$ and $\varpi$ is tangent to $X$ on $\beta \cap S.$  The normal field $\varpi$ is said to be {\em orientable} respect to $X$ if for any component $\alpha \subset C_S$ having endpoints $P_1,$ $P_2$ lying in $\partial(M_S),$    the basis $B_i=\{(X|_\alpha)'(s_i), \varpi(s_i)\}$ of the tangent plane of $X|_{M_S}$ at $P_i,$  $i=1,2,$ are both positive or negative (with respect to the orientation of $\mathcal{N}$), where $s_i$ is the value of the arclength parameter $s$ for which  $\alpha(s_i)=P_i,$ $i=1,2.$

The following objects will play a crucial role in the statement of our approximation results by minimal surfaces (see Theorem \ref{co:immaprox} in Section \ref{sec:aproxi}).

\begin{definition}
We call  $\mathcal{ M}_\ggot^*(S)$ as the  space of marked immersions  $X_\varpi:=(X,\varpi),$ where $X \in \mathcal{ M}_\ggot(S)$ and $\varpi$ is an  orientable smooth normal field along $C_S$ respect to $X.$
\end{definition}


Given $X_\varpi \in \mathcal{ M}_\ggot^*(S),$ let $\partial
X_\varpi=(\hat{\phi}_j)_{j=1,2,3}$ be the complex vectorial
``1-form"  on $S$ given by  $\partial
X_\varpi|_{M_S}=\partial (X|_{M_S}),$ $\partial X_\varpi(\alpha'(s))= dX
(\alpha'(s)) + i \varpi(s),$ where $\alpha$ is a component of
$C_S$ and $s$ is the arclength parameter of $X|_\alpha$  for
which $\{dX (\alpha'(s_i)), \varpi(s_i)\}$ are positive, where as above
$s_1$ and $s_2$ are the values of $s$ for which  $\alpha(s) \in \partial(M_S).$ If $(U,z=x+i y)$ is a
conformal chart on $\mathcal{N}$ such that $\alpha \cap U=z^{-1}(\r \cap
z(U)),$ it is clear that $(\partial X_\varpi)|_{\alpha \cap
U}=\big[dX (\alpha'(s)) + i \varpi(s)\big]s'(x)dz|_{\alpha \cap
U},$ hence
 $\partial X_\varpi \in \Omega_\hgot(S)^3.$ Furthermore, $\hat{g}=\hat{\phi}_3/(\hat{\phi}_1-i\hat{\phi}_2)\in \mathcal{ F}_\mgot(S)$
 provided that $\hat{g}^{-1}(\infty)\subset S^\circ.$

Obviously, $\hat{\phi}_j$ is smooth on $S,$ $j=1,2,3,$ and the
same occurs for $\hat{g}.$ Notice that $\sum_{j=1}^3 \hat{\phi}_j^2=0,$ $\sum_{j=1}^3 |\hat{\phi}_j|^2$ never vanishes on $S$ and $\mbox{Re} (\hat{\phi}_j)$ is an ``exact" real 1-form on $S,$
$j=1,2,3,$ hence we also have $X(P)=X(Q)+\mbox{Re} \int_{Q}^P
(\hat{\phi}_j)_{j=1,2,3},$ $P,$ $Q \in S.$
 For these reasons, $(\hat{g},\hat{\phi}_3)$ will be called as the generalized ``Weierstrass data" of $X_\varpi.$
As $X|_{M_S} \in \mathcal{ M} (M_S),$  then $(\phi_j)_{j=1,2,3}:=(\hat{\phi}_j|_{M_S})_{j=1,2,3},$ and $g:=\hat{g}|_{M_S}$ are  the  Weierstrass data and the meromorphic Gauss map of $X|_{M_S},$ respectively.

The space $\mathcal{M}_\ggot^*(S)$ is naturally endowed with the following $\Ccal^1$ topology:

\begin{definition}\label{def:C1}
Given $X_{\varpi},$ $Y_{\xi}\in\mathcal{M}_\ggot^*(S),$ we set
\[
\|X_{\varpi}-Y_{\xi}\|_{1,S}:=\|X-Y\|_{0,S}+\big\|\partial X_{\varpi}-\partial Y_{\xi}\big\|_{0,S} \;\;(\text{see \eqref{eq:norma0}}).
\]

Given $F \in \Mcal(S)$, we denote by $\varpi_F$ the conormal field of $F$ along  $C_S.$ Notice that $(\partial F)|_S=\partial F_{\varpi_F},$ where $F_{\varpi_F}:=(F|_S,\varpi_F) \in \Mcal_\ggot^*(S).$ If $F,$ $G\in \Mcal(S),$ we set $$\|F-X_{\varpi}\|_{1,S}:=\|F_{\varpi_F}-X_{\varpi}\|_{1,S}\quad \text{and} \quad \|F-G\|_{1,S}:=\|F_{\varpi_F}-G_{\varpi_G}\|_{1,S} .$$
\end{definition}

\begin{definition} Let $W$ be an open subset of $\mathcal{N}$ containing $S.$ We shall say that a marked immersion $X_\varpi \in \mathcal{M}_\ggot^*(S)$ can be  approximated in the $\Ccal^1$ topology on $S$ by conformal minimal immersions in $\mathcal{M}(W)$ if for any $\epsilon>0$ there exists  $Y\in  \mathcal{M}(W)$ such that $\|Y-X_\varpi\|_{1,S}<\epsilon.$ 
\end{definition}

The group homomorphism  $$p_{X_\varpi}:\mathcal{ H}_1(S,\z) \to \r^3, \quad p_{X_\varpi}(\gamma)=\mbox{Im} \int_\gamma \partial X_\varpi,$$ is said to be the {\em generalized flux map} of $X_\varpi.$ Obviously,  $p_{X_{\varpi_Y}}=p_Y|_{\mathcal{ H}_1(S,\z)}$ provided that  $X=Y|_{S}.$


\section{The Approximation Lemmas} \label{sec:aproxi}

The aim of this section is to obtain an approximation result for marked minimal immersions on admissible subsets by minimal immersions defined on an arbitrary larger domain of finite topology (see Theorem \ref{co:immaprox} below). 

Throughout this section,  $W$  will denote a domain of finite topology in $ \mathcal{ N}$ and $S$ an  admissible compact subset contained in $W.$

Several extensions of classical Runge-Mergelyan theorems can be found in \cite{roy,sche1,sche2}. For our purposes, we need only the following compilation result:

\begin{theorem} \label{th:runge} For any  $f \in \mathcal{ F}_\mgot(S)$ and integral divisor $D \in \div(S)$ with  $\mbox{supp}(D)\subset S^\circ,$ there exists $\{f_n\}_{n \in \n} \in \mathcal{ F}_\mgot(W)$ such that $f_n|_S-f \in \Fcal_\hgot(S)$  and $\big(f_n|_S-f\big)_0 \geq D$ for all $n,$ and $\{\|f_n|_S-f\|_{0,S}\}_{n \in \n} \to 0.$
\end{theorem}

We start with the following
\begin{lemma} \label{lem:funaprox}
Consider $f \in \mathcal{ F}_\mgot(S)$ such that $f$ never vanishes on $S-S^\circ(=\partial(M_S) \cup C_S).$

Then there exists    $\{f_n\}_{n \in \n} \subset \mathcal{ F}_\mgot(W)$ satisfying that $f_n|_S-f \in \Fcal_\hgot(S)$ and $(f_n)=(f)$ on $W$ for all $n,$ and $\{\|f_n|_S-f\|_{0,S}\}_{n \in \n} \to 0.$  In particular, $f_n$ is  holomorphic and never vanishing on $W-S$ for all $n.$
\end{lemma}

\begin{proof} Let $\mu$ and $b$ denote the genus of $W$ and the number of topological ends of  $W-\mbox{supp}(({f})).$
It is well known (see \cite{farkas}) that there exist $2 \mu+b-1$
cohomologically independent 1-forms  in $\Omega_\mgot(W) \cap
\Omega_\hgot(W-\mbox{supp}((f)))$  generating the first holomorphic De
Rham cohomology group  $\mathcal{H}^1_{\text{hol}}(W-\mbox{supp}((f))).$ Furthermore, the 1-forms can be chosen having at most single poles at points of $\mbox{supp}((f)).$ Thus, the map
$\mathcal{H}^1_{\text{hol}}(W-\mbox{supp}((f))) \to \c^{2
\mu+b-1}$,  $ \tau \mapsto \left( \int_{c} \tau \right)_{c \in
B_0},$ where $B_0$ is any homology basis of $W-\mbox{supp}((f)),$
is a linear isomorphism. By hypothesis, $\mbox{supp}((f))\subset S^\circ$ and $df/f \in \Omega_\mgot(S).$ Thus, there exists $\tau\in \Omega_\mgot(W) \cap
\Omega_\hgot\big(W-\mbox{supp}((f))\big)$ with  single poles at points of $\mbox{supp}((f))$  such that $\frac{1}{2 \pi i}
\int_\gamma \tau \in \z$ for all $\gamma \in \mathcal{ H}_1
\big(W-\mbox{supp}((f)),\z\big)$ and $df/f-\tau \in \Omega_\hgot
(S)$ is exact. 

Set $f_0=f e^{-\int \tau}.$ Since $\log (f_0)\in \mathcal{ F}_\hgot(S)$ then $f_0 \in \mathcal{ F}_\hgot(S)$ and it  
 never vanishes on $S.$  By Theorem
\ref{th:runge}, there exists $\{h_n\}_{n \in \n}
\subset \mathcal{ F}_\hgot(W)$ such that  $\{\|h_n|_S-\log (f_0)\|_{0,S}\}_{n \in \n} \to 0.$ It suffices to take $f_n=e^{h_n+\int \tau}$ for all $n.$
\end{proof}

\begin{lemma} \label{lem:formaprox}
Consider $\theta \in \Omega_\mgot(S)$  never vanishing on $S-S^\circ.$

Then  there exists $\{\theta_n\}_{n \in \n} \in \Omega_\mgot(W)$ satisfying that $\theta_n-\theta\in \Omega_\hgot(S)$ and $(\theta_n)=(\theta)$ on $W,$ and $\{\|\theta_n|_S-\theta\|_{0,S}\}_{n \in \n} \to 0.$   In particular, $\theta_n$ is  holomorphic and never vanishing on $W-S$ for all $n.$
\end{lemma}

\begin{proof} First of all, notice that there exists $\tau\in \Omega_\hgot(W)$ with finitely many zeroes. Indeed, since $W$ has finite topology and up to elementary surgery operations, we  can be view $W$ as an open domain in a non-simply connected compact Riemann surface $\hat{W},$ $\partial(\hat{W})=\emptyset.$ It suffices to take a non-identically zero holomorphic 1-form $\hat{\tau}$ on $\hat{W}$ and set $\tau=\hat{\tau}|_W.$

Label $f=\theta/\tau \in \mathcal{ F}_\mgot(S).$ By Lemma \ref{lem:funaprox}, there exists
$\{f_n\}_{n \in \n}$ in $\mathcal{ F}_\mgot(W)$ such that  $\{\|f_n|_S-f\|_{0,S}\}_{n \in \n} \to 0$ and  $(f_n)=(f)$
on $W$ for all $n.$ It suffices to set $\theta_n:=f_n \tau$ for all $n\in \n.$
\end{proof}

The following lemma is the kernel of this section. It is the key for proving Theorem \ref{co:immaprox}.  

\begin{lemma}[The Approximation Lemma] \label{lem:weiaprox}
Let $\Phi=(\phi_j)_{j=1,2,3}$ be a smooth triple in $\Omega_\hgot(S)^3$ such that $\sum_{j=1}^3 \phi_j^2=0$ and $\sum_{j=1}^3 |\phi_j|^2$ never vanishes on $S.$  Then  $\Phi$ can be approximated in the $\Ccal^0$ topology on $S$ by a sequence $\{\Phi_n=(\phi_{j,n})_{j=1,2,3}\}_{n \in \n}\subset \Omega_\hgot(W)^3$ satisfying that:
\begin{enumerate}[(i)]
  \item $\sum_{j=1}^3 \phi_{j,n}^2=0$ and $\sum_{j=1}^3 |\phi_{j,n}|^2$ never vanishes on $W,$
  \item  $\Phi_{n}-\Phi$ is exact on $S$ for all $n.$
\end{enumerate}
\end{lemma}

\begin{proof} Label $g=\frac{\phi_3}{\phi_1-i \phi_2},$  $\eta_1=\frac{1}{g} \phi_3=\phi_1-i \phi_2$ and $\eta_2=g \phi_3=-\phi_1-i \phi_2,$ and  notice that $\eta_1,$ $\eta_2 \in \Omega_\hgot(S).$ 

Let $\mathcal{ B}_S$ be a homology basis of $\mathcal{ H}_1(S,\z),$ and label $\nu\in \n$ as the number of elements in $\mathcal{B}_S.$

 The following two claims reduce the proof to a more comfortable setting. 

\begin{assertion} \label{ass:gauss}
Without loss of generality, we can assume that $g|_{M_S}$ is not constant.
\end{assertion}
\begin{proof}
Suppose for a moment that $g|_{M_S}$ is constant, and up to
replacing $\Phi$ by $\Phi \cdot A$ for a suitable orthogonal
matrix $A \in \mathcal{ O}(3,\r),$  assume that $g \neq \infty.$
For each $h \in \mathcal{F}_\hgot(W),$ set $\eta_2(h)=(g+h)^2 \eta_1$
and $\phi_3(h)=\eta_1 (g+h).$  Consider the holomorphic map
$\mathcal{T}:\mathcal{F}_\hgot(W) \to \c^{2 \nu},$
$\mathcal{T}(h)=(\int_c(\eta_2(h)-\eta_2,\phi_3(h)-\phi_3))_{c \in
\mathcal{B}}.$ Note that $\mathcal{T}^{-1}(0)$ is conical, that is to say,  if
$\mathcal{T}(h)=0$ then $\mathcal{T}(\lambda h)=0$ for all
$\lambda \in \c.$ Furthermore, since $\mathcal{F}_\hgot(W)$ has infinite dimension  we can choose a non-constant $h \in \mathcal{T}^{-1}(0).$ Take  $\{\lambda_n\}_{n \in \n}\subset \c$   converging to zero, set $h_n:=\lambda_n h\in \mathcal{T}^{-1}(0)$ for all $n,$ and notice that $\{h_n\}_{n \in \n}\to 0$ in the $\Ccal^0$ topology on $S.$

Set
$\Psi_n\equiv (\psi_{1,n},\psi_{2,n},\psi_{3,n}):= (\frac{1}{2}
(\eta_1-\eta_2(h_n)),\frac{i}{2}(\eta_1+\eta_2(h_n)),\phi_3(h_n))
\in \Omega_\hgot(S)^3,$ and observe that
$\sum_{j=1}^3 \psi_{j,n}^2=0,$ $\sum_{j=1}^3 |\psi_{j,n}|^2$ never
vanishes on $S$ and  $g_n=\frac{\psi_{3,n}}{\psi_{1,n}-i \psi_{2,n}}$
is holomorphic and non-constant on $M_S,$ $n$ large enough (without loss of generality, for all $n$). Since $\mathcal{T}(h_n)=0,$ it is clear that $\Psi_n-\Phi$ is exact on $S,$  $n
\in \n.$ If the lemma holds for $\Psi_n$ for all $n,$ we can
construct a  sequence $\{\hat{\Psi}_{n,m}\}_{m \in \n} \subset
\Omega_\hgot(S)^3$ converging to $\Psi_n$ in the $\Ccal^0$ topology
on $S$ and satisfying that $\hat{\Psi}_{n,m}-\Psi_n$ is exact on
$S$ for all $n.$ A standard diagonal argument proves the claim.
\end{proof}

\begin{assertion} \label{ass:motion}
Without loss of generality,  we can assume that  $g,$ $1/g$ and $d g$  never vanish on $\partial(M_S) \cup C_S$ (hence the same holds for $\eta_i,$  $i=1,2,$ and $\phi_j,$ $j=1,2,3$). In particular, $g  \in \mathcal{ F}_\mgot(S)$ and  $d g \in \Omega_\mgot(S).$
\end{assertion}
\begin{proof} Take a sequence $M_{1} \supset M_{2} \supset \ldots $ of compact regions in $W$ such that $M_n^\circ$ is a tubular neighborhood of $M_S$ in $W$ for all $n,$  $M_{n} \subset M_{n-1}^\circ$ for any $n,$ $\cap_{n \in  \n} M_{n}=M_S,$ $\Phi$ holomorphically extends  (with the same name) to $M_{1},$  $\sum_{j=1}^3 |\phi_j|^2\neq 0$ on $M_1,$ and  $g,$ $1/g,$  and $d g$  never vanish on $\partial({M}_{n})$ for all $n$  (take into account Claim \ref{ass:gauss}). Choose $M_n$ in such a way that $S_n:=M_n \cup C_S\subset W$ is  an admissible subset and $\gamma-M_n^\circ$ is a (non-empty) Jordan arc for any component $\gamma$ of $C_S.$ In particular, 
$C_{S_n}= C_S-M_n^\circ,$ $n \in \n.$

Let  $(h_n,\psi_{3,n})\in \mathcal{ F}_\mgot(S_n) \times \Omega_\hgot(S_n)$ be any smooth data  such that
\begin{itemize}
  \item  $(h_n,\psi_{3,n})|_{M_{S_n}}=(g,\phi_3)|_{M_{S_n}}$ and $\sum_{j=1}^3 |\psi_{j,n}|^2$ never vanishes on $S_n,$ where $\Psi_n=(\psi_{j,n})_{j=1,2,3}=\big(\frac{1}{2}(1/h_n-h_n),\frac{i}{2}(1/h_n+h_n),1\big) \psi_{3,n} \in \Omega_\hgot(S_n)^3,$ $n \in \n,$
  \item  $h_n,$ $1/h_n$ and $d h_n$  never vanish on $\partial(M_{S_n}) \cup C_{S_n},$
  \item $\Psi_n|_S-\Phi$ is exact on $S,$ and
  \item the sequence $\{\Psi_n|_S\}_{n \in \n} \subset \Omega_\hgot(S)^3$ converges to $\Phi$ in the $\Ccal^0$ topology on $S.$
\end{itemize}
The existence of such data follows from classical approximation results by smooth functions.

Label $\mathcal{ T}\subset \Omega_\hgot(W)^3$ as the subspace of data $\Psi$ formally satisfying $(i)$ and $(ii)$ in the statement of the lemma. If the lemma held for any of the data in $\{\Psi_n\,|\; n \in \n\},$   $\Psi_n$ would lie in the closure of $\mathcal{ T}$ in $\Omega_\hgot(S_n)^3$ with respect to the $\Ccal^0$ topology on $S_n$ for all  $n \in \n.$ By a standard diagonal argument again, the same would occur for  $\Phi$ and we are done.
\end{proof}

Consider the period map
$$\mathcal{ P}:\mathcal{ F}_\hgot(W)\times \mathcal{ F}_\hgot(W) \to \c^{3 \nu},\quad
\mathcal{ P}((h_1,h_2))=\big( \int_{c} ((e^{h_2-h_1}-1)\eta_1, (e^{h_2+h_1}-1) \eta_2, (e^{h_2}-1) \phi_3) \big)_{c \in \mathcal{ B}_S}.$$
The meromorphic data inside the integrals are the difference between the Weierstrass data on $S$ associated to $( e^{h_1} g,e^{h_2}\phi_3)$ and the ones associated to $(g,\phi_3).$ The Weierstrass data determined by  $( e^{h_1} g,e^{h_2}\phi_3)$ satisfy (i), and if in addition $\mathcal{ P}((h_1,h_2))=0$ then also (ii).

The first key step in the proof of the lemma is to show that the Implicit Function Theorem can be applied to $\mathcal{P}$ at $(0,0).$ To do this, endow $\mathcal{ F}_\hgot(S)$ with the  maximum norm,  and observe that $\mathcal{P}$ is Fr\'{e}chet differentiable. It suffices to check that the Fr\'{e}chet derivative $\mathcal{ A}_0$ of $\mathcal{P}$ at $(0,0)$ has maximal rank.

\begin{assertion} \label{ass:regular}
$\mathcal{ A}_0:\mathcal{ F}_\hgot(W)\times \mathcal{ F}_\hgot(W)\to \c^{3 \nu}$ is surjective.
\end{assertion}

\begin{proof} Reason by contradiction and assume that $\mathcal{ A}_0(\mathcal{ F}_\hgot(W)\times \mathcal{ F}_\hgot(W))$ lies in a complex subspace
$\mathcal{ U}=\{\big((x_c,y_c,z_c)\big)_{c \in \mathcal{ B}_S} \in \c^{3 \nu}\,|\; \sum_{c \in \mathcal{ B}_S} \big(A_c x_c+B_c y_c +D_c z_c\big)=0\},$ where
$A_c,$ $B_c$ and $D_c \in \c$ for all $c\in \mathcal{ B}_S$  and $\sum_{c \in \mathcal{ B}_S} \big(|A_c|+|B_c|+|D_c|\big)\neq 0.$ This simply means that:

\begin{equation} \label{eq:fun}
-  \int_{\Gamma_1}  h \eta_1 +\int_{\Gamma_2} h \eta_2 =\int_{\Gamma_1}  h \eta_1 +\int_{\Gamma_2} h \eta_2+\int_{\Gamma_3} h \phi_3=0
\end{equation} for all $h \in \mathcal{ F}_\hgot(W),$ where $\Gamma_1= \sum_{c \in \mathcal{ B}_S} A_c c,$ $\Gamma_2= \sum_{c \in \mathcal{ B}_S} B_c c$ and
$\Gamma_3= \sum_{c \in \mathcal{ B}_S} D_c c.$

Label $\Sigma_0=\{f \in \mathcal{ F}_\hgot(W) \,|\; (f) \geq (\phi_3)^2\}.$ By Theorem \ref{th:runge}, the function $h=df/\phi_3 \in \mathcal{ F}_\hgot(S)$ lies in the closure of $\mathcal{ F}_\hgot(W)$ in the $\Ccal^0$ topology on $\mathcal{ F}_\hgot(S)$ for any $f \in \Sigma_0.$ Therefore, equation (\ref{eq:fun}) can be applied formally to $h=df/\phi_3,$ getting that $\int_{\Gamma_1} \frac{1}{g} df=\int_{\Gamma_2} g df=0$ for all $f \in \Sigma_0.$ Integrating by parts,
\begin{equation} \label{eq:fun1}
\int_{\Gamma_1} f \frac{d g}{g^2} =\int_{\Gamma_2} f\,dg =0
\end{equation} for all $f \in \Sigma_0.$

Let us show that $\Gamma_1=0.$

Let $\mu$ and $b$ denote the genus of $W$ and the number of ends of  $W.$
It is well known (see \cite{farkas}) that there exist $2 \mu+b-1$ cohomologically independent 1-forms  in $\Omega_\hgot(W)$  generating the first holomorphic De Rham cohomology group  $\mathcal{ H}^1_{\text{hol}}(W)$  of $W.$ Thus, the map $\mathcal{H}^1_{\text{hol}}(W) \longrightarrow \c^{2 \mu+b-1}$,  $ \tau \mapsto \left( \int_{c} \tau \right)_{c \in B_0},$ where $B_0$ is any homology basis of $W,$  is a linear isomorphism. Assume that $\Gamma_1\neq 0$ and take $[\tau] \in H^1_{\text{hol}}(W)$ such that $\int_{\Gamma_1} \tau \neq 0.$ Since $W$ is an open surface,   $ {\mathcal{ F}_\hgot}(W)$ has infinite dimension and we can find $F \in  {\mathcal{ F}_\hgot}(W)$ such that
$(\tau+dF)_0 \geq  (dg)_0 (g)_\infty^2 (\phi_3)^2.$ Set $h:=\frac{(\tau+dF)g^2}{d g}$ and note that $(h) \geq (\phi_3)^2.$ By Theorem \ref{th:runge},  $h$ lies in the closure of $\Sigma_0$ in $\mathcal{ F}_\hgot(S)$ with respect to the $\Ccal^0$ topology, hence equation (\ref{eq:fun1}) can be formally applied to $h,$  giving that $ \int_{\Gamma_1} \tau+dF=\int_{\Gamma_1} \tau=0,$ a contradiction.

By a similar argument $\Gamma_2=0$ and equation (\ref{eq:fun}) becomes:
\begin{equation} \label{eq:fun2}
\int_{\Gamma_3} h \phi_3=0
\end{equation} for all $h \in \mathcal{ F}_\hgot(W).$

Since $\sum_{c \in \mathcal{ B}_S} \big(|A_c|+|B_c|+|D_c|\big) \neq 0,$ then $\Gamma_3 \neq 0.$
Reason as above and choose   $[\tau] \in H^1_{\text{hol}}(W)$ and $F \in  {\mathcal{ F}_\hgot}(W)$ such that $\int_{\Gamma_3} \tau \neq 0$  and
$(\tau+dF)_0 \geq  (\phi_3).$ Set $h:=\frac{\tau+dF}{\phi_3}$ and note that $h\in \mathcal{ F}_\hgot(S).$ By Theorem \ref{th:runge},  $h$ lies in the closure of $\mathcal{ F}_\hgot(W)$ in $\mathcal{ F}_\hgot(S)$ with respect to the $\Ccal^0$ topology,  and equation (\ref{eq:fun2}) gives that $ \int_{\Gamma_3} \tau+dF=\int_{\Gamma_3} \tau=0,$ a contradiction. This proves the claim.
\end{proof}

Let $\{e_1,\ldots,e_{3 \nu}\}$ be a basis of $\c^{3 \nu},$  fix
$H_i=(h_{1,i},h_{2,i}) \in \mathcal{ A}_0^{-1}(e_i)$ for all $i,$ and set $\mathcal{
Q}_0:\c^{3 \nu} \to \c^{3 \nu}$ as the analytical map given by
$$\mathcal{ Q}_0((z_i)_{i=1,\ldots,3 \nu})=\mathcal{ P}(\sum_{i=1,\ldots,3
\nu} z_i H_i).$$ By Claim \ref{ass:regular} $d (\mathcal{ Q}_0)_0$ is
an isomorphism, so there exists a closed Euclidean ball $U\subset
\c^{3 \nu}$ centered at the origin such that $\mathcal{ Q}_0:U \to
\mathcal{ Q}_0(U)$ is an analytical diffeomorphism. Furthermore,
notice that $0=\mathcal{ Q}_0(0) \in \mathcal{ Q}_0(U)$ is an interior
point of  $\mathcal{ Q}_0(U).$

On the other hand, by Lemmas \ref{lem:funaprox} and
\ref{lem:formaprox} there exists a sequence $\{(f_n,\psi_n)\}_{n
\in \n} \subset \mathcal{ F}_\mgot(W)\times \Omega_\hgot(W)$ such that
$(f_n)=(g)$ and $(\psi_n)=(\phi_3)$ for all $n,$ and
$\{(f_n,\psi_n)|_S\}_{n \in \n} \to (g,\phi_3)$ in the
$\Ccal^0$ topology on $S.$

Label $\mathcal{ P}_n:\mathcal{ F}_\hgot(W)\times \mathcal{ F}_\hgot(W) \to \c^{3 \nu}$ as the Fr\'{e}chet differentiable map
$$\mathcal{ P}_n((h_1,h_2))=\big( \int_{c} (e^{h_2-h_1}\eta_{1,n}-\eta_1, e^{h_2+h_1} \eta_{2,n}-\eta_2, e^{h_2} \psi_n-\phi_3) \big)_{c \in \mathcal{ B}_S},$$ where $\eta_{1,n}=\frac{1}{2} \psi_n(1/f_n-f_n)$ and
$\eta_{2,n}=\frac{i}{2} \psi_n(1/f_n+f_n),$ and call $\mathcal{ Q}_n:\c^{3 \nu} \to \c^{3 \nu}$ as the analytical map $\mathcal{ Q}_n((z_i)_{i=1,\ldots,3 \nu})=\mathcal{ P}_n(\sum_{i=1,\ldots,3 \nu} z_i H_i)$  for all $n \in \n.$ Since $\{\mathcal{ Q}_n\}_{n \in \n} \to \mathcal{ Q}_0$  uniformly on compacts subsets of $\c^{3 \nu},$ without loss of generality we can suppose that $\mathcal{ Q}_n:U \to \mathcal{ Q}_n(U)$ is an analytical diffeomorphism  and  $0 \in \mathcal{ Q}_n(U)$ for all $n.$ Label ${\bf{y_n}}=(y_{1,n},\ldots,y_{3 \nu,n})$ as the unique point in $U$ such that $\mathcal{ Q}_n ({\bf{y_n}})=0$ and note that $\{{\bf{y_n}}\}_{n \in \n} \to 0.$ Setting
$$g_n=e^{\sum_{j=1}^{3 \nu} y_{j,n} h_{1,j}} f_n,\quad  \phi_{3,n}=e^{\sum_{j=1}^{3 \nu} y_{j,n} h_{2,j}} \psi_n$$ for all $n\in \n,$ the sequence $\{(g_n,\phi_{3,n})\}_{n \in\n }$  solves the lemma.
\end{proof}

The proof of the following corollary is just an elementary adjustment of the above one.

\begin{corollary} \label{co:weiaprox} In the previous lemma we can choose $\phi_{3,n}=\phi_3$ for all $n \in \n,$ provided that $\phi_3$  extends  holomorphically to  $W$ and $\phi_3$ never vanishes on $C_S.$
\end{corollary}
\begin{proof} Without loss of generality, we can suppose that $g,$ $1/g$ and $dg$ never vanish on $\partial(M_S) \cup C_S.$ Indeed, consider a sequence  $\{M_{n}\}_{n \in \n}$ like in the proof of Claim \ref{ass:motion}. We have that $\cap_{n=1}^\infty M_n=M_S,$   $S_n:=M_n \cup C_S$ is admissible in $W,$  $g,$ $1/g$ and $dg$ never vanishes on $\partial(M_{n})$ and $\gamma-M_n^\circ$ is a (non-empty) Jordan arc for any component $\gamma$ of $C_S,$ for all $n\in \n.$ Let  $h_n\in \mathcal{ F}_\mgot(S_n)$ be a smooth datum  such that:
\begin{itemize}
  \item  $h_n|_{M_{S_n}}=g|_{M_{S_n}}$ and $\sum_{j=1}^3 |\psi_{j,n}|^2$ never vanishes on $S_n,$ where $\Psi_n=(\psi_{j,n})_{j=1,2,3}=\big(\frac{1}{2}(1/h_n-h_n),\frac{i}{2}(1/h_n+h_n),1\big) \phi_{3} \in \Omega_\hgot(S_n)^3,$ $n \in \n,$
  \item  $h_n,$ $1/h_n$ and $d h_n$  never vanish on $\partial(M_{S_n}) \cup C_{S_n},$
  \item $\Psi_n|_S-\Phi$ is exact on $S,$ and
  \item the sequence $\{\Psi_n|_S\}_{n \in \n} \subset \Omega_\hgot(S)^3$ converges to $\Phi$ in the $\Ccal^0$ topology on $S.$
\end{itemize}
Reasoning as in the proof of Claim \ref{ass:motion}, if the Lemma held for any of the data in $\{\Psi_n\,|\; n \in \n\}$ the same would occur for  $\Phi$ and we are done.

Reasoning like in the proof of Lemma \ref{lem:weiaprox}, we can prove that $\hat{\mathcal{ A}}_0:\mathcal{ F}_\hgot(W) \to \c^{ \nu}$ is surjective, where $\hat{\mathcal{ A}}_0$ is the Fr\'{e}chet derivative of
 $\hat{\mathcal{ P}}:\mathcal{ F}_\hgot(W) \to \c^{2 \nu},$  $\hat{\mathcal{ P}}(h):=\mathcal{ P}(h,0).$  Then take $\hat{H}_i \in \hat{\mathcal{ A}}_0^{-1}(e_i)$ for all $i$, where $\{e_1,\ldots, e_{2 \nu}\}$ is a basis of $\c^{2 \nu},$ and define $\hat{\mathcal{ Q}}_0:\c^{2 \nu} \to \c^{2 \nu}$ by $\hat{\mathcal{ Q}}_0((z_i)_{i=1,\ldots,2 \nu})=\hat{\mathcal{ P}}(\sum_{i=1,\ldots,2 \nu} z_i \hat{H}_i).$

 Use Riemann-Roch Theorem  to find a holomorphic function $H \in {\mathcal F}_\hgot(W)$ such that $(H)=(\phi_3|_{W-S}),$ and then Lemma \ref{lem:funaprox} to get  $\{f_n\}_{n \in \n} \subset \mathcal{ F}_\mgot(W)$ such that $(f_n)=(g|_{S})$ for all $n$ and $\{f_n|_S\}_{n \in \n} \to g/H$ in the $\Ccal^0$ topology on $S.$

  Set $\hat{\mathcal{ P}}_n:\mathcal{ F}_\hgot(W) \to \c^{2 \nu}$ by $\hat{\mathcal{ P}}_n(h)=\big( \int_{c} (e^{-h}\eta_{1,n}-\eta_1, e^{h} \eta_{2,n}-\eta_2) \big)_{c \in \mathcal{ B}_S},$ where $\eta_{1,n}=\frac{1}{2} \phi_3(\frac{1}{f_n H}-f_n H)$ and
$\eta_{2,n}=\frac{i}{2} \phi_3(\frac{1}{f_n H}+f_n H),$ and call $\hat{\mathcal{ Q}}_n:\c^{2 \nu} \to \c^{2 \nu}$ as the analytical map $\hat{\mathcal{ Q}}_n((z_i)_{i=1,\ldots,2 \nu})=\hat{\mathcal{ P}}_n(\sum_{i=1,\ldots,2 \nu} z_i \hat{H}_i)$  for all $n \in \n.$  To finish, reason as in the  proof of Lemma \ref{lem:weiaprox}.
\end{proof}

As a consequence of Lemma \ref{lem:weiaprox} and Corollary \ref{co:weiaprox}, one has the following approximation result of marked immersions by conformal minimal immersions. It will play a crucial role in the proof of the main results of this paper.

\begin{theorem}\label{co:immaprox}

Let $S\subset \mathcal{N}$ be  admissible and connected, and let $W\subset \mathcal{N}$ be a domain of finite topology containing $S$ such that $i_*:\mathcal{H}_1(S,\z)\to \mathcal{H}_1(W,\z)$ is an isomorphism, where $i:S\to W$ denotes the inclusion map. Let $X_\varpi\in \mathcal{M}_\ggot^*(S),$ and write $X=(X_j)_{j=1,2,3}$ and $\partial X_\varpi=(\hat{\phi}_j)_{j=1,2,3}.$

Then, for any $\xi>0$ there exists $Y\in\mathcal{M}(W)$ such that $p_Y=p_{X_\varpi}$ and $\|Y-X_\varpi\|_{1,S}<\xi.$

Furthermore, if $\hat{\phi}_3$ extends to a holomorphic 1-form on $W$ that never vanishes on $C_S,$ then $Y=(Y_j)_{j=1,2,3}$ can be chosen so that $Y_3|_S=X_3.$ 
\end{theorem}
\begin{proof}
Applying Lemma \ref{lem:weiaprox} and Corollary \ref{co:weiaprox} to the data $(\hat{\phi}_j)_{j=1,2,3},$ and then integrating with the suitable initial condition (take into account that $S$ is connected), we can find a sequence $\{F_n\}_{n\in\n}\subset\mathcal{M}(W)$ such that 
 $\{\|F_n-X_\varpi\|_{1,S}\}_{n \in \n}\to 0$ and
 the flux map $p_{F_n}=p_{X_\varpi}$ for all $n\in\n.$ Furthermore, if $\hat{\phi}_3$ extends to a holomorphic 1-form on $W$ that never vanishes on $C_S,$ then $F_n=(F_{n,j})_{j=1,2,3}$ can be chosen so that $F_{n,3}|_S=X_3$ for all $n\in\n.$ 

It suffices to set $Y:=F_n$ for a large enough $n.$ 
\end{proof}
The strength of this result is that only smoothness is assumed for $X$ on $C_S.$ This provides an enormous capability for modeling minimal surfaces in $\r^3.$ 

\section{Properness and Conformal Structure of Minimal
Surfaces}\label{sec:fun}

The Runge type lemma for minimal surfaces below concentrates most of the technical
computations required in the proof of the main theorem of this
section. Roughly speaking, this lemma asserts that a compact minimal surface whose boundary lies outside a wedge on $\r^3$ can be stretched  near the boundary in such a way that the  boundary of the new surface lies outside a parallel wedge. The strength of this lemma is that in this process the topology and conformal structure of the surface can be arbitrarily  enlarged. Moreover, the flux map of the arising surface can be prescribed. See Figure \ref{fig:lemma}.

From now on, we label $x_k:\r^3 \to \r$ as the $k$-th coordinate
function, $k=1,2,3.$ For each $\theta\in(0,\pi/2),$ $\delta \in \r,$  we call $$\Pi_\delta (\theta)=\{(x_1,x_2,x_3) \in \r^3 \,|\; x_3+\tan(\theta) x_1 > \delta\}.$$

Although Theorem I in the Introduction was stated for $\theta \in (0,\pi/2),$ in the following lemma and for technical reasons we will restrict ourselves to the case $\theta \in (0,\pi/4).$ 
\begin{lemma} \label{lem:fun}
Let $M,$ $V \subset \mathcal{ N}$ be two Runge compact regions with analytical boundary such that $M \subset V^\circ.$

Consider $X \in \mathcal{ M}(M)$ and let $p:\mathcal{ H}_1(V,\z)\to\r$ be any morphism extension of $p_X.$ Suppose there are  $\theta\in (0,\pi/4)$ and $\delta\in (0,+\infty)$ such that $X(\partial(M)) \subset \Pi_\delta(\theta) \cup \Pi_\delta(-\theta).$

Then, for any $\epsilon>0$  there exists $Y \in \mathcal{ M}(V)$ such that 
\begin{enumerate}[\rm \bf (1)]
\item $p_Y=p.$ 
\item $\|Y-X\|_{1,M}<\epsilon.$  
\item $Y(\partial(V)) \subset \Pi_{\delta+1}(\theta) \cup \Pi_{\delta+1}(-\theta).$ 
\item $Y(V-M) \subset \Pi_\delta(\theta) \cup \Pi_\delta(-\theta).$
\end{enumerate}
\end{lemma}

\begin{figure}[ht]
    \begin{center}
    \scalebox{0.5}{\includegraphics{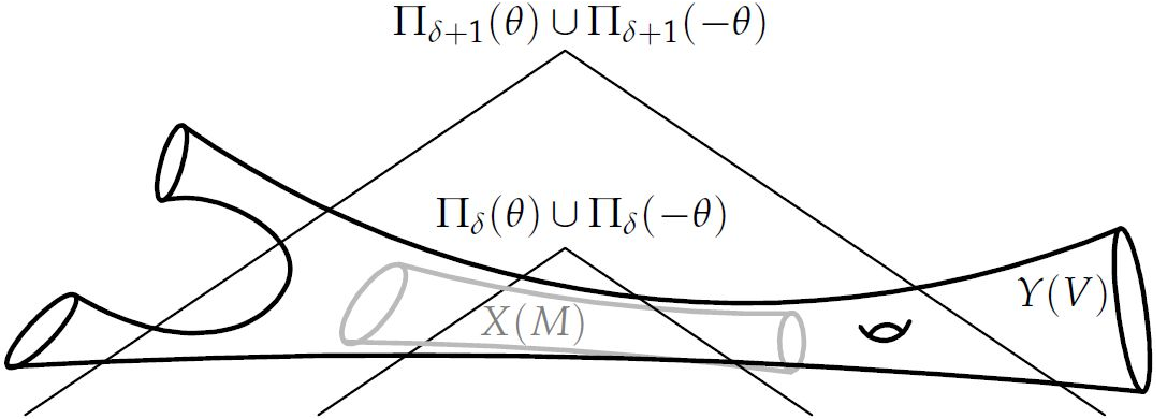}}
        \end{center}
        \vspace{-0.5cm}
\caption{Lemma \ref{lem:fun}}\label{fig:lemma}
\end{figure}

Before proving Lemma \ref{lem:fun} let us show the following particular instance:

\begin{lemma}\label{lem:car0}
Lemma \ref{lem:fun} holds when the Euler characteristic $\chi(V-M^\circ)$ vanishes.
\end{lemma}

\subsection{Proof of Lemma \ref{lem:car0}}

Since $M \subset V^\circ$ and $V^\circ-M$ has no bounded components in $V^\circ,$ then
 $V-M^\circ=\cup_{j=1}^k A_j,$ where $A_1,\ldots, A_k$ are pairwise disjoint compact annuli. Write $\partial(A_j)=\alpha_j \cup \beta_j,$ where $\alpha_j\subset \partial(M)$ and $\beta_j\subset \partial(V)$  for all $j.$
 
First of all, let us introduce some subsets of $V-M^\circ.$ See Figure \ref{fig:anillo}.

Since $X(\partial(M)) \subset \Pi_\delta(\theta) \cup \Pi_\delta(-\theta),$
each $\alpha_j$ can be divided into finitely many compact Jordan arcs $\alpha_j^i,$ $i=1,\ldots, n_j\geq 2,$ laid end to end 
and satisfying that either $X(\alpha_j^i) \subset \Pi_\delta(\theta)$ or  $X(\alpha_j^i) \subset \Pi_\delta(-\theta)$ for all $i.$ Up to refining the partitions, we can assume that $n_j=m \in \n$ for all $j.$  Set $\Ical=\{1,\ldots,m\}\times\{1,\ldots,k\}.$

An arc $\alpha_j^i$ is said to be {\em positive} if $X(\alpha_j^i) \subset \Pi_\delta(\theta),$ and  {\em negative} otherwise. Notice that
$X(\alpha_j^i) \subset \Pi_\delta(-\theta)$ for any negative (and possibly for some positive)
$\alpha_j^i.$ We also label $Q_j^i$ and $Q_j^{i+1}$ as the
endpoints of $\alpha_j^i,$ in such a way that
$Q_j^{i+1}=\alpha_j^i \cap \alpha_j^{i+1},$ $i=1,\ldots,m,$  (obviously, $Q_j^{m+1}=Q_j^1$).

Let $\{r_j^i\,|\;i=1,\ldots,m\}$ be a collection of pairwise
disjoint analytical compact Jordan arcs in $A_j$ such that $r_j^i$ has
initial point $Q_j^i \in \alpha_j,$ final point $P_j^i\in \beta_j,$ $r_j^i$ is otherwise disjoint from $\partial(A_j),$ and $r_j^i$ meets transversally $\alpha_j^i$ at $Q_j^i,$ for all
$i$ and $j.$ As above, $P_j^{m+1}=P_j^1$ and $r_j^{m+1}=r_j^1.$

Let $W$ be a small open tubular neighborhood of $V$ in $\mathcal{ N},$ and notice that $i_*:\mathcal{H}_1(M,\z)\to \mathcal{H}_1(W,\z)$ is an isomorphism, where $i:M\to W$ denotes the inclusion map.

Consider the admissible set $$M_0=M \cup \big( \cup_{(i,j)\in\Ical} r_j^i \big).$$ Call $\Omega_j^i$ as the closed disc in $A_j$
bounded by $\alpha_j^i\cup r_j^i \cup r_j^{i+1}$ and the compact Jordan arc $\beta_j^i\subset\beta_j$ connecting $P_j^i$ and
$P_j^{i+1},$ and containing no $P_j^k$ for $k\neq i,i+1.$ Obviously $\Omega_j^i \cap \Omega_j^{i+1}=r_j^{i+1},$
$i<m,$ $\Omega_j^{m} \cap \Omega_j^1=r_j^{1},$ and
$A_j=\cup_{i=1}^{m} \Omega_j^i.$ The region $\Omega_j^i$ is said to
be positive (respectively, negative) if $\alpha_j^i$ is positive
(respectively, negative).  See Figure \ref{fig:anillo}.
\begin{figure}[ht]
    \begin{center}
    \scalebox{0.35}{\includegraphics{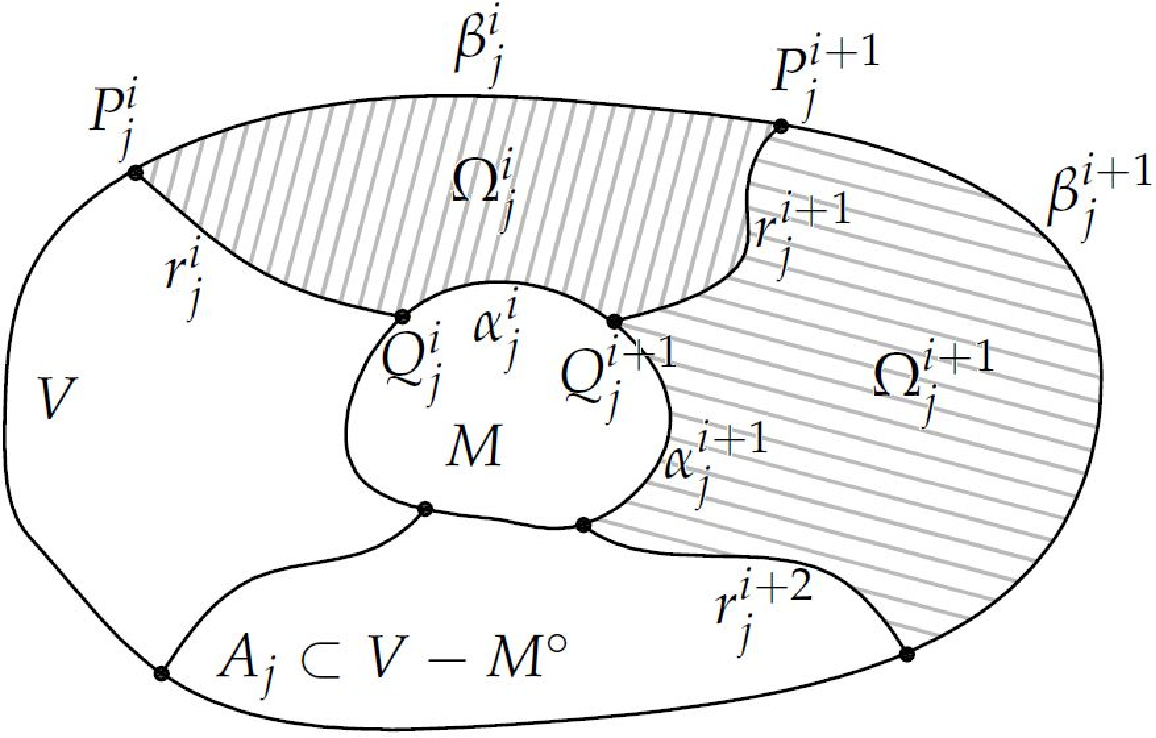}}
        \end{center}
        \vspace{-0.5cm}
\caption{The annulus $A_j.$}\label{fig:anillo}
\end{figure}

Set $\mathcal{ I}_+=\{(i,j)\in\Ical \,|\; 
\Omega_j^i \; \mbox{is positive}\}$ and $\mathcal{ I}_-=\{(i,j)\in\Ical \,|\;
\Omega_j^i \; \mbox{is negative}\}.$ Without loss of generality, and up to a symmetry with respect to the plane $\{x_1=0\},$ we suppose that $\Ical_+ \neq \emptyset.$

The proof consists of three different construction steps. In the first one we construct an immersion $H\in\Mcal(V)$  satisfying the theses of the lemma on $M_0$ (see properties {\rm \bf (1$_{H}$)} to {\rm \bf (4$_{H}$)} below). In the second step we deform $H$ on $\cup_{(i,j)\in\Ical_+}\Omega_j^i$  to obtain $Z\in\Mcal(V)$ satisfying the theses of the lemma on $M_0\cup\big(\cup_{(i,j)\in\Ical_+}\Omega_j^i\big),$ see properties {\rm \bf (1$_{Z}$)} to {\rm \bf (4$_{Z}$)} for details. Finally, in the third step of the proof (which is symmetric to the second one) we  modify $Z$ on $\cup_{(i,j)\in\Ical_-}\Omega_j^i$ to get the immersion $Y\in\Mcal(V)$ which solves the lemma. Each stage preserves the already done in the previous ones.

The {\bf first step of the proof} consists of constructing $H\in\Mcal(V)$ satisfying, among other properties, the theses of the lemma {\em just on $M_0.$} To be more precise, $H$ will satisfy that:
\begin{enumerate}[\rm \bf (1$_{H}$)]
\item $p_{H}=p.$
\item $\|H-X\|_{1,M}<\epsilon/3.$
\item $H(P_j^i) \in \Pi_{\delta+1}(\theta) \cap \Pi_{\delta+1}(-\theta)$ for all $(i,j)\in\Ical.$
\item 
$\displaystyle\left\{
\begin{array}{ll}
\text{\rm \bf (4$_{H}^+$)} & \text{$H(r_j^{i}\cup \alpha_j^i \cup  r_j^{i+1}) \subset \Pi_\delta(\theta)$ for all $(i,j)\in\Ical_+.$}\\
\text{\rm \bf (4$_{H}^-$)} & \text{$H(r_j^{i}\cup \alpha_j^i \cup  r_j^{i+1}) \subset \Pi_\delta(-\theta)$ for all $(i,j)\in\Ical_-.$}
\end{array}
\right.
$
\end{enumerate}
In particular, if $(i,j)\in \Ical_+$ and $(i+1,j)\in \Ical_-$ then $ H(r_j^{i+1})\subset \Pi_\delta(\theta)\cap \Pi_\delta(-\theta).$

We proceed as follows. Take $\hat{X}\in \mathcal{ M}_\ggot(M_0)$ such that $\hat{X}|_M=X,$  and
\begin{equation}\label{eq:c(P)}
\hat{X}(P_j^i)\in \Pi_{\delta+1}(\theta) \cap \Pi_{\delta+1}(-\theta)\quad \text{for all $(i,j)\in \Ical.$}
\end{equation}
In addition, choose $\hat{X}$
in such a way that:
\begin{equation}\label{eq:1+1-}
\text{$\hat{X}(r_j^i\cup  r_j^{i+1})\subset \Pi_\delta(\theta)$ for all $(i,j)\in \Ical_+$}\quad\text{and}\quad
\text{$\hat{X}(r_j^i\cup  r_j^{i+1})\subset \Pi_\delta(-\theta)$ for all $(i,j)\in \Ical_-.$}
\end{equation}
\begin{figure}[ht]
    \begin{center}
    \scalebox{0.55}{\includegraphics{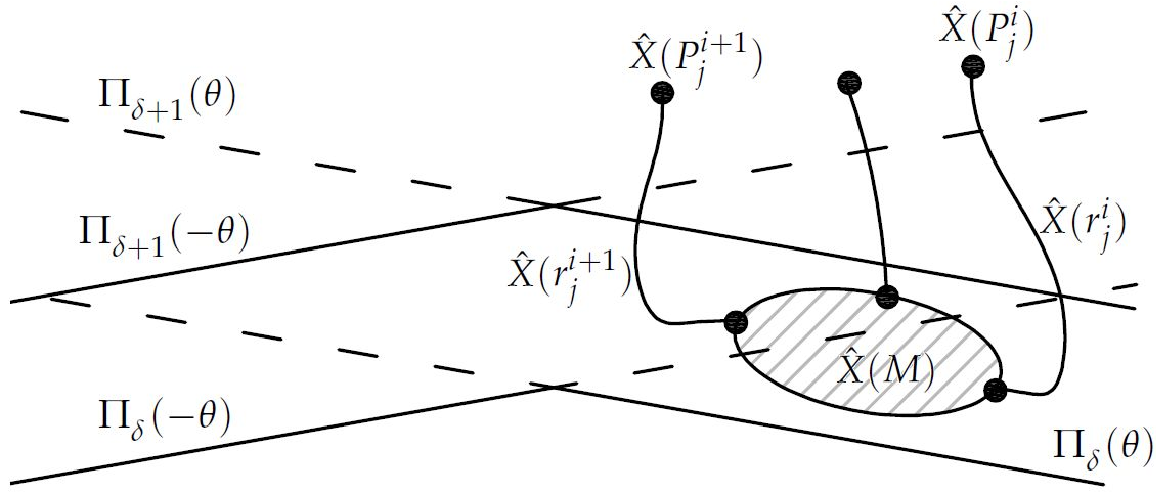}}
        \end{center}
        \vspace{-0.5cm}
\caption{$\hat{X}(M_0).$}\label{fig:step1}
\end{figure}
See Figure \ref{fig:step1}. The existence of  such a $\hat{X}$ is elementary. 
 Choose any arbitrary smooth normal field $\varpi_0$ along $C_{M_0}=\cup_{(i,j)\in\Ical}r_j^i$ respect to $\hat{X}$ so that $\hat{X}_{\varpi_0}\in\mathcal{M}_\ggot^*(M_0).$ Applying Theorem \ref{co:immaprox} to 
$\hat{X}_{\varpi_0},$ $W$ and a small enough $\xi\in(0,\epsilon/3)$ one can find $H \in \mathcal{
M}(V)$ such that
$p_{H}=p_{\hat{X}_{\varpi_0}}=p_X=p$ (hence {\bf (1$_H$)} holds) and
\begin{equation}\label{eq:Y0-X0}
 \|H-\hat{X}_{\varpi_0}\|_{1,M_0}<\xi<\epsilon/3\quad\text{(hence {\bf (2$_H$)} holds).}
\end{equation}
Properties {\bf (3$_H$)} and {\bf (4$_H$)}  follow from \eqref{eq:c(P)}, \eqref{eq:1+1-} and \eqref{eq:Y0-X0} provided that $\xi$ is chosen small enough. This concludes the construction of $H.$

In the {\bf second step of the proof}, we will deform $H$ hardly on $M\cup\big(\cup_{(i,j) \in \mathcal{ I}_-} \Omega_j^i\big)$ and strongly on $V-\big[M\cup\big(\cup_{(i,j) \in \mathcal{ I}_-} \Omega_j^i\big)\big]$ to obtain a new immersion $Z\in\Mcal(V).$ This deformation will preserve the coordinate function $x_3+\tan(\theta) x_1,$ that is to say,
\[
(x_3+\tan(\theta) x_1)\circ H=(x_3+\tan(\theta) x_1)\circ Z\quad\text{on $V.$} 
\] 
Furthermore, $Z$ will satisfy the theses of Lemma \ref{lem:fun} {\em just on} $M_0\cup(\cup_{(i,j) \in \mathcal{ I}_+} \Omega_j^i).$ 
To be more precise, $Z$ will satisfy that:
\begin{enumerate}[\rm \bf (1$_Z$) ]
\item $p_{Z}=p_H=p.$
\item $\|Z-H\|_{1,M}<\epsilon/3.$
\item $\displaystyle\left\{
\begin{array}{ll}
\text{\rm \bf (3$_Z^+$)} & \text{$Z(\beta_j^i) \subset \Pi_{\delta+1}(\theta) \cup \Pi_{\delta+1}(-\theta)$ for all $(i,j)\in\Ical_+.$}\\
\text{\rm \bf (3$_Z^-$)} & \text{$Z(\{P_j^i,P_j^{i+1}\}) \subset \Pi_{\delta+1}(\theta) \cap \Pi_{\delta+1}(-\theta)$ for all $(i,j)\in\Ical_-.$}
\end{array}
\right.
$

\item $\displaystyle\left\{
\begin{array}{ll}
\text{\rm \bf (4$_Z^+$)} & \text{$Z(\Omega_j^i) \subset \Pi_\delta(\theta) \cup \Pi_\delta(-\theta)$ for all  $(i,j) \in \Ical_+.$}\\
\text{\rm \bf (4$_Z^-$)} & \text{$Z(r_j^{i}\cup \alpha_j^i \cup  r_j^{i+1}) \subset \Pi_\delta(-\theta)$ for all $(i,j)\in\Ical_-.$}
\end{array}
\right.
$
\end{enumerate}

In order to construct $Z$ and for a simpler writing, it will be convenient to rotate $H$  as follows. Let $L^+:\r^3 \to \r^3$ denote
the counterclockwise rotation of angle $\theta$ around the straight line parallel to the $x_2$-axis and containing the
point $(0,0,\delta).$ As $\theta\in(0,\pi/4)$ then 
\begin{multline}\label{eq:giro+}
L^+(\Pi_\delta(\theta))=\Pi_\delta(0),\quad  L^+(\Pi_\delta(-\theta))=\Pi_\delta(-2\theta),\\
L^+(\Pi_{\delta+1}(\theta))=\Pi_{\delta_1}(0)\quad\text{and}\quad L^+(\Pi_{\delta+1}(-\theta))=\Pi_{\delta_2}(-2\theta),
\end{multline}
where $\delta_1=\delta+\cos(\theta)$ and $\delta_2=\cos(\theta)+\sin(\theta)\cot(2\theta).$

Call $H^+=(H^+_j)_{j=1,2,3}:=L^+ \circ H \in \mathcal{ M}(V),$ and notice that $H_3^+=(x_3+\tan(\theta) x_1)\circ H$ on $V.$

For any $(i,j)\in\mathcal{ I}_+,$ let $K_j^i$ be a closed disc with analytical boundary in $\Omega_j^i$ such that $K_j^i \cap \partial(\Omega_j^i)$ consists of a (non-empty) compact Jordan arc in $\beta_j^i-\{P_j^i,P_j^{i+1}\},$
\begin{equation}\label{eq:a1'}
H^+(\beta_j^i-K_j^i)\subset L^+\big(\Pi_{\delta+1}(\theta) \cap \Pi_{\delta+1}(-\theta)\big)\subset \Pi_{\delta_1}(0)
\end{equation}
(that is to say, $H_3^+>\delta_1$ on $\beta_j^i-K_j^i$), and
\begin{equation}\label{eq:a1}
H^+(\overline{\Omega_j^i-K_j^i}) \subset  \Pi_\delta(0)
\end{equation}
(that is to say, $H_3^+>\delta$ on $\overline{\Omega_j^i-K_j^i}$).
This choice is possible from {\bf (3$_H$)}, {\bf (4$_H^+$)}, \eqref{eq:giro+} and a continuity argument, provided that $K_j^i$ is chosen large enough in $\Omega_j^i.$ See Figures \ref{fig:omega} and \ref{fig:step2}.

\begin{figure}[ht]
    \begin{center}
    \scalebox{0.3}{\includegraphics{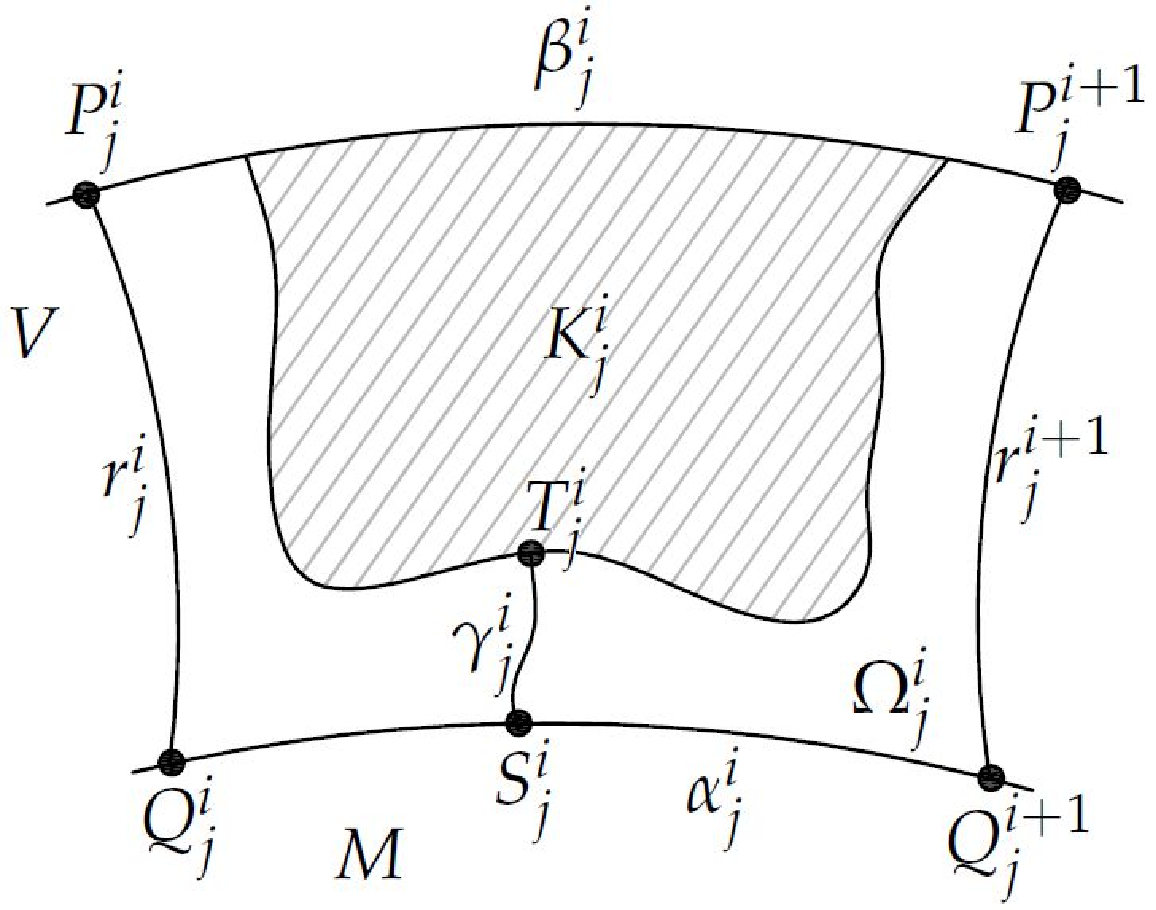}}
        \end{center}
        \vspace{-0.5cm}
\caption{The set $\Omega_j^i.$}\label{fig:omega}
\end{figure}

Since $-2\theta\in (-\pi/2,0)$ and $H^+(\cup_{(i,j) \in \mathcal{ I}_+} K_j^i)$ is compact, there exists $\lambda^+>0$ such that
\begin{equation}\label{eq:lambda+}
(-\lambda^+,0,0)+H^+(\cup_{(i,j) \in \mathcal{ I}_+} K_j^i)\subset \Pi_{\delta_2}(-2\theta) = L^+(\Pi_{\delta+1}(-\theta)).
\end{equation}

The key idea in this deformation stage is to {\em push} $(H^+_1,H^+_3)(\cup_{(i,j)\in\mathcal{ I}_+} K_j^i)\subset\r^2$ to the left in the direction of the $x_1$-axis a distance $\lambda^+,$ while preserving $H_3^+$ on $V$ (see Figure \ref{fig:step2}) and hardly modifying $H^+$ on $M\cup\big(\cup_{(i,j) \in \mathcal{ I}_-} \Omega_j^i\big).$ In this way we obtain a new immersion $Z^+\in \Mcal (V)$ such that $x_3\circ Z^+=H_3^+$ on $V$ and $Z^+(\cup_{(i,j)\in\Ical_+} K_j^i)\subset L^+(\Pi_{\delta+1}(-\theta)).$   By \eqref{eq:a1'}, \eqref{eq:a1} and \eqref{eq:lambda+},  $Z:= (L^+)^{-1} \circ Z^+$  will satisfy the desired properties. No matter the values of both $x_2\circ Z^+$ on $V-\big[M\cup\big(\cup_{(i,j) \in \mathcal{ I}_-} \Omega_j^i\big)\big]$ and $x_1\circ Z^+$ on $V-\big[M\cup\big(\cup_{(i,j) \in \mathcal{ I}_+} K_j^i\big)\cup\big(\cup_{(i,j) \in \mathcal{ I}_-} \Omega_j^i\big)\big].$

\begin{figure}[ht]
    \begin{center}
    \scalebox{0.65}{\includegraphics{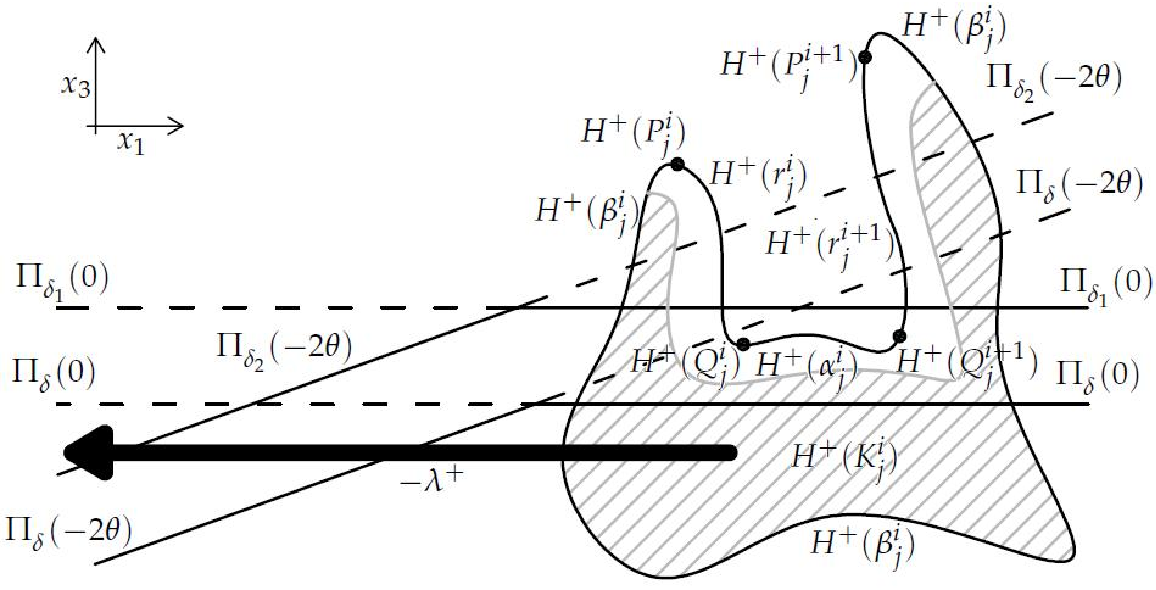}}
        \end{center}
        \vspace{-0.5cm}
\caption{$H^+(\Omega_j^i),$ $(i,j)\in\Ical_+,$ and the second deformation stage.}\label{fig:step2}
\end{figure}

To carry out this deformation, we have to introduce a suitable admissible set and marked immersion on it. 
For any $(i,j)\in\Ical_+,$ let $\gamma_j^i$ be a compact analytical Jordan arc in $\Omega_j^i$ satisfying that the endpoints $S_j^i$ and $T_j^i$  of $\gamma_j^i$ lie in $\alpha_j^i-\{Q_j^i,Q_j^{i+1}\}$
 and $\partial(K_j^i)-\beta_j^i,$ respectively, and $\gamma_j^i$
 is otherwise disjoint from $K_j^i \cup \partial(\Omega_j^i).$ Without loss of generality, we can suppose that $\gamma_j^i$ and $\alpha_j^i$ (resp., $\partial(K_j^i)$) meet transversally at $S_j^i$ (resp., $T_j^i$) and $\partial H_3^+$ never vanishes on $\gamma_j^i.$ See Figure \ref{fig:omega}. Consider the admissible set
$$S_+=\big(M\cup\big(\cup_{(i,j) \in \mathcal{ I}_-} \Omega_j^i\big)
\big)  \cup \big( \cup_{(i,j) \in \mathcal{ I}_+} (K_j^i \cup
\gamma_j^i)\big),$$ and notice that $M_{S_+}=M\cup\big(\cup_{(i,j) \in \mathcal{ I}_-} \Omega_j^i\big)\cup\big(\cup_{(i,j) \in \mathcal{ I}_+} K_j^i\big)$ and $C_{S_+}=\cup_{(i,j)\in \Ical_+} \gamma_j^i.$  

\begin{assertion} \label{ass:inter}
There exists $\hat{H}^+_{\varpi_+}\in \mathcal{M}_\ggot^*(S_+),$ where $\hat{H}^+=(\hat{H}^+_j)_{j=1,2,3},$ such that 
\begin{enumerate}[{\rm (i)}]
\item $\hat{H}^+= H^+$ on $M\cup\big(\cup_{(i,j) \in \mathcal{ I}_-} \Omega_j^i\big),$
\item $\hat{H}^+_1=H^+_1-\lambda^+$  on $\cup_{(i,j) \in \mathcal{I}_+} K_j^i,$ 
\item $\hat{H}^+_3= H^+_3$ and $(\partial \hat{H}^+_{\varpi_+})_3=\partial H^+_3$ on $S_+.$
\end{enumerate}
\end{assertion}

\begin{proof} Call $\psi_3^+:=\partial H_3^+,$ and write $g^+$ for the meromorphic Gauss map of $H^+.$

Consider any smooth $\hat{g}^+\in \mathcal{F}_\mgot(S)$ such that: $\hat{g}^+=g^+$ on $M_{S_+},$   $\hat{g}^+$ never vanishes on   $C_{S^+}$ and $\frac1{2}\mbox{Re}\big(\int_{\gamma_{i,j}} (1/\hat{g}^+-\hat{g}^+) \psi_3^+ \big)=H_1^+(T_{i,j})-H_1^+(S_{i,j})-\lambda^+,$ where we have oriented $\gamma_j^i$ with initial point $S_{i,j}$ and final point $T_{i,j},$ $(i,j) \in \Ical_+.$ The existence of a such $\hat{g}^+$ follows from the surjectivity of the continuous map $$J:\mathcal{V} \to \r^{\nu_+},\quad  J(h)=\big( \mbox{Re} \big(\int_{\gamma_j^i} (1/h-h)\psi_3^+ \big) \big)_{(i,j) \in \Ical_+},$$ 
where $\mathcal{V}$ is the space   $\{h \in \mathcal{F}_\mgot(S)\,|\;$ $h$ is smooth, $h=g^+$ on $M_{S_+}$ and $h$ never vanishes on $C_{S_+}\}$ endowed with the $\Ccal^0$ topology on $S,$ and $\nu_+$ is the number of elements of $\Ical_+.$

If suffices to set $\hat{H}^+:S_+ \to \r^3,$ 
$\hat{H}^+=H^+(P_0^+)+\int_{P_0^+} \hat{\Psi}^+,$ and $\varpi_+(s):=\mbox{Im}\big(\hat{\Psi}^+ (({\gamma_j^i})'(s))\big),$ where
$$\hat{\Psi}^+= \big(\frac1{2} (1/\hat{g}^+-\hat{g}^+),\frac{i}{2}(1/\hat{g}^++\hat{g}^+),1\big)\psi_3^+,$$ $P_0^+ \in M_{S_+},$ and 
$s$ is the arclength parameter along $\gamma_j^i,$ $(i,j) \in \Ical_+.$ 
\end{proof}

Applying Theorem \ref{co:immaprox} to $\hat{H}^+_{\varpi_+},$ $W$ and a small enough $\xi\in(0,\epsilon/3),$ there exists $Z^+\in \mathcal{ M}(V)$ such that 
\begin{equation}\label{eq:23}
\|Z^+-\hat{H}^+_{\varpi_+}\|_{1,S_+}<\xi<\epsilon/3,
\end{equation}
\begin{equation}\label{eq:pXn+}
p_{Z^+}=p_{\hat{H}^+_{\varpi_+}}=p_{H^+}
\end{equation} 
and $x_3\circ Z^+|_{S_+}=x_3\circ \hat{H}^+.$ In particular, $x_3\circ Z^+=x_3\circ H^+$ on $V$ (see Claim \ref{ass:inter}-(iii)).

%
%
%

Then, one has:
\begin{enumerate}[{\rm ({a}1)}]
\item $\|Z^+-H^+\|_{1,M}<\epsilon/3.$ See \eqref{eq:23} and Claim \ref{ass:inter}-(i).

\item $Z^+\big(\cup_{(i,j)\in \mathcal{ I}_+}\overline{\Omega_j^i-K_j^i}\big) \subset \Pi_\delta(0).$ Indeed, by \eqref{eq:a1} one has $H^+(\cup_{(i,j)\in \mathcal{ I}_+}\overline{\Omega_j^i-K_j^i}) \subset \Pi_\delta(0)=\{x_3>\delta\}.$ Since $x_3\circ Z^+=x_3\circ H^+,$ then the inclusion holds. 
  
\item $Z^+\big(\cup_{(i,j)\in \mathcal{ I}_+}(\beta_j^i -K_j^i)\big)\subset \Pi_{\delta_1}(0).$ Indeed, by \eqref{eq:a1'} one infers that $H^+(\cup_{(i,j)\in \mathcal{ I}_+}(\beta_j^i -K_j^i))\subset \Pi_{\delta_1}(0)=\{x_3>\delta_1\}.$ Since $x_3\circ Z^+=x_3\circ H^+,$ we are done.
\end{enumerate}
Furthermore, if $\xi>0$ is chosen small enough then, from \eqref{eq:23},
\begin{enumerate}[{\rm ({a}1)}]
\item[(a4)] $Z^+(\cup_{(i,j) \in \mathcal{ I}_+} K_j^i)\subset \Pi_{\delta_2}(-2\theta).$  Take into account \eqref{eq:lambda+} and Claim \ref{ass:inter}-(ii).

\item[(a5)] $Z^+(\{P_j^i,P_j^{i+1}\})\subset \Pi_{\delta_2}(-2\theta) \cap \Pi_{\delta_1}(0),$ for any $(i,j)\in\mathcal{ I}_-.$ Use {\bf (3$_H$)}, \eqref{eq:giro+} and Claim \ref{ass:inter}-(i).

\item[(a6)] $Z^+(r_j^{i}\cup \alpha_j^i \cup  r_j^{i+1})\subset \Pi_{\delta}(-2\theta)$ for any $(i,j)\in\mathcal{
  I}_-.$ It follows from {\bf (4$_H^-$)}, \eqref{eq:giro+} and Claim \ref{ass:inter}-(i).
\end{enumerate}

Taking into account \eqref{eq:giro+}, it is not hard to check that the immersion $Z:=(L^+)^{-1}\circ Z^+\in\Mcal(V)$ satisfies the desired properties. Indeed, {\bf (1$_Z$)}, {\bf (2$_Z$)}, {\bf (3$_Z^-$)} and {\bf (4$_Z^-$)} follow from \eqref{eq:pXn+}, \eqref{eq:23}, (a5) and (a6), respectively. Finally, {\bf (3$_Z^+$)} follows from (a3) and (a4), whereas  {\bf (4$_Z^+$)} follows from (a2) and (a4).

This concludes the second step of the proof.

The {\bf third step of the proof} is symmetric to the second one. We will deform $Z$ hardly on $M\cup\big(\cup_{(i,j) \in \mathcal{ I}_+} \Omega_j^i\big)$ and strongly on $V-\big[M\cup\big(\cup_{(i,j) \in \mathcal{ I}_+} \Omega_j^i\big)\big].$ Now we will preserve the coordinate function $(x_3-\tan(\theta) x_1)\circ Z$ on $V.$ The so-arising immersion $Y\in\Mcal(V)$ will be the solution of the lemma. To be more precise, it will verify
\begin{enumerate}[\rm \bf (1$_Y$)]
\item $p_{Y}=p_Z=p.$
\item $\|Y-Z\|_{1,M}<\epsilon/3.$
\item $Y(\partial V)\subset \Pi_{\delta+1}(\theta) \cup \Pi_{\delta+1}(-\theta).$
\item $Y(V-M) \subset \Pi_\delta(\theta) \cup \Pi_\delta(-\theta).$
\end{enumerate}

\begin{remark}
If $\Ical_-=\emptyset,$ it suffices to set $Y:=Z$ and notice that properties {\rm \bf (1$_Y$)} to {\rm \bf (4$_Y$)} follow directly from {\rm \bf (1$_Z$)},  {\rm \bf (2$_Z$)},  {\rm \bf (3$^+_Z$)} and {\rm \bf (4$^+_Z$)} above.
\end{remark}
Assume that $\Ical_-\neq \emptyset,$ and let us construct $Y.$ 

First, set $L^-:=(L^+)^{-1}$ and observe that 
\begin{multline}\label{eq:giro-}
L^-(\Pi_\delta(\theta))=\Pi_\delta(2\theta),\quad L^-(\Pi_\delta(-\theta))=\Pi_\delta(0),\\ L^-(\Pi_{\delta+1}(\theta))=\Pi_{\delta_2}(2\theta) \quad\text{and}\quad L^-(\Pi_{\delta+1}(-\theta))=\Pi_{\delta_1}(0).
\end{multline}
Denote by $Z^-:=L^-\circ Z \in \mathcal{ M}(V),$ and notice that $x_3\circ Z^-=(x_3-\tan(\theta) x_1)\circ Z$ on $V.$ 

Taking into account \eqref{eq:giro-}, properties {\bf (3$_Z$)} and {\bf (4$_Z$)} can be rewritten in terms of $Z^-$ as follows:
\begin{enumerate}[\rm \bf (1$_{Z^-}$)]
\item[{\bf (3$_{Z^-}$)}] $\displaystyle\left\{\begin{array}{ll}
\text{\rm \bf (3$_{Z^-}^+$)} & \text{$Z^-(\beta_j^i) \subset \Pi_{\delta_2}(2\theta) \cup \Pi_{\delta_1}(0)$ for all $(i,j)\in\Ical_+.$}\\
\text{\rm \bf (3$_{Z^-}^-$)} & \text{$Z^-(\{P_j^i,P_j^{i+1}\}) \subset \Pi_{\delta_2}(2\theta) \cap \Pi_{\delta_1}(0)$ for all $(i,j)\in\Ical_-.$}
\end{array}
\right.
$
\item[{\bf (4$_{Z^-}$)}] $\displaystyle\left\{\begin{array}{ll}
\text{\rm \bf (4$_{Z^-}^+$)} & \text{$Z^-(\Omega_j^i) \subset \Pi_\delta(2\theta) \cup \Pi_\delta(0)$ for all  $(i,j) \in \Ical_+.$}\\
\text{\rm \bf (4$_{Z^-}^-$)} & \text{$Z^-(r_j^{i}\cup \alpha_j^i \cup  r_j^{i+1}) \subset \Pi_\delta(0)$ for all $(i,j)\in\Ical_-.$}
\end{array}
\right.
$
\end{enumerate}

%
%
%
%
%
%
%

For any $(i,j)\in\mathcal{ I}_-$ let $K_j^i$ be a closed disc with analytical boundary in
$\Omega_j^i$ such that $K_j^i \cap \partial(\Omega_j^i)\neq \emptyset$ consists of a compact Jordan arc in $\beta_j^i-\{P_j^i,P_j^{i+1}\},$
\begin{equation}\label{eq:a1-'}
Z^-(\beta_j^i-K_j^i)\subset \Pi_{\delta_2}(2\theta) \cap \Pi_{\delta_1}(0)\subset \Pi_{\delta_1}(0)
\end{equation}
and
\begin{equation}\label{eq:a1-}
Z^-(\overline{\Omega_j^i-K_j^i}) \subset \Pi_\delta(0).
\end{equation}
This choice is possible from {\bf (3$_{Z^-}^-$)}, {\bf (4$_{Z^-}^-$)} and a continuity argument, provided that $K_j^i$ is chosen large enough in $\Omega_j^i.$

Since $2\theta\in(0,\pi/2)$ and $Z^-(\cup_{(i,j) \in \mathcal{ I}_-} K_j^i)$ is compact, then there exists $\lambda^->0$ such that
\begin{equation}\label{eq:lambda-}
(\lambda^-,0,0)+Z^-(\cup_{(i,j) \in \mathcal{ I}_-} K_j^i)\subset \Pi_{\delta_2}(2\theta).
\end{equation}

Now the idea is to  push $(Z^-_1,Z^-_3)(\cup_{(i,j)\in\mathcal{ I}_-} K_j^i)\subset\r^2$ to the right in the direction of the $x_1$-axis a distance $\lambda^-,$ while preserving $Z^-_3$ on $V$ and hardly modifying $Z^-$ on $M\cup\big( \cup_{(i,j) \in \mathcal{ I}_+} \Omega_j^i\big).$ By \eqref{eq:a1-'}, \eqref{eq:a1-} and \eqref{eq:lambda-}, the arising immersion $Y^-$ will satisfy the desired properties (up to composing with $L^+,$ we get $Y$). This time, no matter the values of both $x_2\circ Z^-$ on $V-\big[M\cup\big( \cup_{(i,j) \in \mathcal{ I}_+} \Omega_j^i\big)\big]$ and $x_1\circ Z^-$ on $V-\big[M\cup\big( \cup_{(i,j)\in \Ical_-}K_j^i\big)\cup\big( \cup_{(i,j) \in \mathcal{ I}_+} \Omega_j^i\big)\big].$ 

We proceed as in the previous step.

For any $(i,j)\in\Ical_-,$ let $\gamma_j^i$ be a compact analytical Jordan arc in $\Omega_j^i$ satisfying that the endpoints $S_j^i$ and $T_j^i$  of $\gamma_j^i$ lie in $\alpha_j^i-\{Q_j^i,Q_j^{i+1}\}$
 and $\partial(K_j^i)-\beta_j^i,$ respectively, and $\gamma_j^i$
 is otherwise disjoint from $K_j^i \cup \partial(\Omega_j^i).$ Without loss of generality, we can assume that $\gamma_j^i$ and $\alpha_j^i$ (resp., $\partial(K_j^i)$) meet transversally at $S_j^i$ (resp., $T_j^i$) and $\partial Z_3^-$ never vanishes on $\gamma_j^i.$  Consider the admissible subset
$$S_-=M\cup\big( \cup_{(i,j) \in \mathcal{ I}_+} \Omega_j^i
\big)  \cup \big( \cup_{(i,j) \in \mathcal{ I}_-} (K_j^i \cup
\gamma_j^i)\big).$$ 

Following the arguments in Claim \ref{ass:inter} one can prove the following 
\begin{assertion}\label{ass:inter-}
There exists $\hat{Z}^-_{\varpi_-}\in \mathcal{M}_\ggot^*(S_-),$ where $\hat{Z}^-=(\hat{Z}^-_j)_{j=1,2,3},$ such that 
\begin{enumerate}[{\rm (i)}]
\item $\hat{Z}^-= Z^-$ on $M\cup\big( \cup_{(i,j) \in \mathcal{ I}_+} \Omega_j^i\big),$
\item $\hat{Z}^-_1=Z^-_1+\lambda^-$  on $\cup_{(i,j) \in \mathcal{I}_-} K_j^i,$
\item $\hat{Z}^-_3= Z^-_3$ and $(\partial \hat{Z}^-_{\varpi_-})_3=\partial Z^-_3$ on $S_-.$
\end{enumerate}
\end{assertion}

Again by Theorem \ref{co:immaprox} applied to $\hat{Z}^-_{\varpi_-},$ $W$ and a small enough $\xi\in(0,\epsilon/3),$ there exists $Y^-\in\mathcal{M}(V)$ such that
\begin{equation}\label{eq:33}
\|Y^- -\hat{Z}^-_{\varpi_-}\|_{1,S_-}<\xi<\epsilon/3,
\end{equation}
\begin{equation}\label{eq:pFn-}
p_{Y^-}=p_{\hat{Z}^-_{\varpi_-}}=p_{Z^-}=L^-\circ p
\end{equation}
and $x_3\circ Y^-|_{S_-}=x_3\circ \hat{Z}^-,$ hence $x_3\circ Y^-=x_3\circ Z^-$ on $V$ (see Claim \ref{ass:inter-}-(iii)).


Arguing as above, if $\xi$ is small enough one has
\begin{enumerate}[{\rm ({b}1)}]
\item $\|Y^--Z^-\|_{1,M}<\epsilon/3.$ Use \eqref{eq:33} and Claim \ref{ass:inter-}-(i).

\item $Y^- \big( \cup_{(i,j)\in \mathcal{ I}_-}\overline{\Omega_j^i-K_j^i} \big)\subset \Pi_{\delta}(0).$ Use that $x_3\circ Y^-=x_3\circ Z^-$ on $V$ and \eqref{eq:a1-}.
  
\item $Y^-\big( \cup_{(i,j)\in \mathcal{ I}_-}(\beta_j^i -K_j^i) \big)\subset\Pi_{\delta_1}(0).$ Use that $x_3\circ Y^-=x_3\circ Z^-$ on $V$ and \eqref{eq:a1-'}.

\item $Y^-( \cup_{(i,j) \in \mathcal{ I}_-} K_j^i )\subset \Pi_{\delta_2}(2\theta).$ Take into account \eqref{eq:lambda-}, \eqref{eq:33} and Claim \ref{ass:inter-}-(ii).
\item $Y^-( \cup_{(i,j)\in \mathcal{ I}_+}\beta_j^i)\subset \Pi_{\delta_2}(2\theta)\cup\Pi_{\delta_1}(0).$ It is implied by {\bf (3$_{Z^-}^+$)} and \eqref{eq:33}.

\item $Y^-(\cup_{(i,j)\in \mathcal{I}_+}\Omega_j^i)\subset \Pi_\delta(0)\cup \Pi_\delta(2\theta).$ See {\bf (4$_{Z^-}^+$)} and \eqref{eq:33}.

\end{enumerate}

Taking into account \eqref{eq:giro-}, it is easy to check that the immersion $Y:=L^+\circ Y^-\in\mathcal{ M}(V)$ satisfies the desired properties. Indeed, properties {\bf (1$_{Y}$)} and {\bf (2$_{Y}$)}     follow from  \eqref{eq:pFn-} and (b1), respectively. Property  {\bf (3$_{Y}$)} follows from (b3), (b4) and (b5), whereas {\bf (4$_{Y}$)} from (b2), (b4) and (b6). This concludes the third step of the proof.

To check that $Y$ {\bf solves the lemma}, observe that {\bf (1$_{Y}$)} proves {\bf (1)}, {\bf (2$_{H}$)}, {\bf (2$_{Z}$)} and {\bf (2$_{Y}$)} imply {\bf (2)},  {\bf (3$_{Y}$)}$=${\bf (3)} and {\bf (4$_{Y}$)}$=${\bf (4)}.  This concludes the proof of Lemma \ref{lem:car0}.

\subsection{Proof of Lemma \ref{lem:fun}}

Since $M$ is Runge, then for any component $C$ of $V-M^\circ$ one has $\partial(C)\cap\partial(V)\neq\emptyset.$ In particular $V-M^\circ$ does not contain closed discs and $-\chi(V-M^\circ)\in\n\cup\{0\}.$ 

The proof goes by induction on
$-\chi(V-M^\circ).$ Lemma \ref{lem:car0} shows the basis of the induction: the result holds for $-\chi(V-M^\circ)=0.$ To check the inductive step, assume that 
Lemma \ref{lem:fun} holds for $-\chi(V-M^\circ)=m\in\n\cup\{0\}$ and let us prove it for $-\chi(V-M^\circ)= m+1.$

Since $-\chi(V-M^\circ)=m+1>0,$ there exists  an
analytic Jordan curve $\hat{\gamma}\in \mathcal{ H}_1(V,\z)-\mathcal{
H}_1(M,\z)$ intersecting $V-M^\circ$ in a Jordan arc $\gamma$ with endpoints $P_1,P_2\in
\partial (M)$ and otherwise disjoint from $\partial (M).$ Without loss of generality we can assume that $\gamma$ matches smoothly with $M,$ and so $M\cup\gamma$ is admissible.
Consider
$F_\varpi\in\mathcal{ M}_\ggot^*(M\cup \gamma)$ such that $F|_M=X,$
$F(\gamma)\subset \Pi_\delta(\theta) \cup \Pi_\delta(-\theta),$ and
$p_{F_\varpi}(\hat{\gamma})=p(\hat{\gamma}).$ Here, we have taken into account
that $F(P_i)\in \Pi_\delta(\theta) \cup \Pi_\delta(-\theta),$ for $i=1,2.$
Notice that $p_{F_\varpi}=p|_{\mathcal{ H}_1(M\cup\gamma,\z)}.$ 

Let $W\subset V^\circ$ be a small open tubular neighborhood of $M\cup\gamma$ in $\mathcal{ N}.$ Notice that $i_*:\mathcal{H}_1(M\cup\gamma,\z)\to \mathcal{H}_1(W,\z)$ is an isomorphism, where $i:M\cup\gamma\to W$ is the inclusion map. Applying Theorem \ref{co:immaprox} to
$F_\varpi,$ $S=M\cup\gamma$ and  $W,$ we can find a compact region $M'$ with non-empty analytical boundary and a minimal immersion
$Z\in\mathcal{ M}(M')$  such that 
\begin{itemize}
\item $M\cup\gamma \subset (M')^\circ
\subset M'\subset W\subset V^\circ,$ $j_*:\mathcal{H}_1(M\cup\gamma,\z)\to \mathcal{H}_1(M',\z)$ is an isomorphism, where $j:M\cup\gamma\to M'$ is the inclusion map,    $-\chi(V-(M')^\circ)= m,$ $M'$ is Runge in $\mathcal{ N},$
\item $\|Z-X\|_{1,M} <\epsilon/2,$  $Z(\partial(M'))\subset \Pi_\delta(\theta) \cup \Pi_\delta(-\theta)$ and  $p_{Z}=p|_{\mathcal{ H}_1(M',\z)}.$
\end{itemize}
 Then, applying the induction hypothesis
to $M',$ $V,$ $Z,$ $\delta,$ $\theta$ and $\epsilon/2,$ we obtain an
immersion $Y\in \mathcal{ M}(V)$ which satisfies the conclusion of the
Lemma.

The proof is done.


\subsection{Main Theorem}

Now we can prove the main theorem of this section.

\begin{theorem} \label{th:fun}
Let $p:\mathcal{ H}_1(\mathcal{ N},\z) \to \r^3$ and $\theta$ be a group morphism  and a real number in  $(0,\pi/2),$ respectively.
Let $M \subset \mathcal{ N}$ be a Runge compact region, and consider a non-flat  $Y\in \mathcal{ M}(M)$  satisfying that $p_{Y}=p|_{\mathcal{ H}_1(M,\z)}$ and $(x_3+\tan (\theta) |x_1|) \circ Y>1.$

Then for any $\epsilon > 0$ there exists a conformal minimal immersion $X:\mathcal{ N} \to \r^3$ satisfying the following properties:
\begin{itemize}
  \item $p_X=p,$
  \item $(x_3+\tan (\theta) |x_1|)\circ X:\mathcal{ N}\to\r$ is a positive proper function, and
  \item $\|X-Y\|_{1,M}< \epsilon.$
\end{itemize}
\end{theorem}
\begin{proof} Without loss of generality, we can assume that $\epsilon<1$ and $\theta \in (0,\pi/4).$

Let $\{M_n\,|\; n \in \n\}$ be an exhaustion  of $\mathcal{ N}$ by Runge compact regions with analytical boundary satisfying that $M_1=M$ and $M_n \subset M_{n+1}^\circ$ $\forall n\in\n.$

Label $Y_1=Y,$ and by Lemma \ref{lem:fun} and  an inductive
process, construct a sequence $\{Y_n\}_{n \in \n}$ of minimal
immersions and a sequence $\{\epsilon_n\}_{n\in\n}$ of positives satisfying that
\begin{enumerate}[(a)]
  \item $Y_n \in \mathcal{ M}(M_n)$ for all $n \in \n,$
  \item $\|Y_{n}-Y_{n-1}\|_{1,M_{n-1}}<\epsilon_n$ for all $n \geq 2,$ where
\[
\epsilon_n= \frac1{2^n}\min\left\{ \epsilon\,,\, \min\left\{\|\partial Y_k\|_{0,M_k}\;|\; k=1,\ldots,n-1\right\}\right\}>0
\]
(notice that $\|\partial Y_k\|_{0,M_k}>0$ since $Y_k$ is an immersion),
  \item $p_{Y_n}=p|_{\mathcal{ H}_1(M_n,\z)}$ for all $n \in \n,$
  \item $ Y_{n}(\partial(M_{n}))\subset \Pi_n(\theta) \cup \Pi_n(-\theta)$   and $ Y_{n}(M_{n}-M_{n-1})\subset \Pi_{n-1}(\theta) \cup \Pi_{n-1}(-\theta)$  for all $n \geq 2.$
\end{enumerate}

By items $(a)$ and $(b)$ and Harnack's theorem, $\{Y_n\}_{n\in
\n}$ uniformly converges on compact subsets of $\mathcal{ N}$ to a
conformal minimal (possibly branched) immersion  $X:\mathcal{ N} \to
\r^3.$ 

Let us check that $(x_3+\tan (\theta) |x_1|)\circ X$ is positive and proper.
Indeed, from (b) one has  $\|X-Y_n\|_{1,M_n} \leq \epsilon/2^n$ for all $n.$ In
particular from $(d)$ we have that $(x_3+ \tan(\theta) |x_1|)
\circ X \geq n-1-\epsilon/2^{n-1}$ on $M_{n}-M_{n-1},$ for all $n \geq 2.$
On the other hand, $(x_3+ \tan(\theta) |x_1|) \circ X \geq
1-\epsilon>0$ on $M_1,$  and so $(x_3+\tan (\theta) |x_1|)\circ X$ is a
positive proper function on $\mathcal{ N}.$

To show that $X$ is an immersion it suffices to check that $\|\partial X\|_{0,M_m}>0$ for all $m\in\n.$ Taking into account (b),
\begin{eqnarray*}
\|\partial X\|_{0,M_m} & \geq & \|\partial Y_{m}\|_{0,M_m}-\sum_{k>m} \|\partial Y_k-\partial Y_{k-1}\|_{0,M_m}\\
 & \geq & \|\partial Y_{m}\|_{0,M_m}-\sum_{k>m} \|Y_k-Y_{k-1}\|_{1,M_{k-1}}\\
 & > &  \|\partial Y_{m}\|_{0,M_m}-\sum_{k>m} \epsilon_k\\
 & > &  \|\partial Y_{m}\|_{0,M_m}-\sum_{k>m} \frac{1}{2^k}\|\partial Y_{m}\|_{0,M_m}\\
 & = & (1-\sum_{k>m}\frac1{2^k})\|\partial Y_{m}\|_{0,M_m}>0.
\end{eqnarray*}

Finally,  item $(c)$ gives that $p_X=p$ and we are done.
\end{proof}

\section{Proper Minimal Surfaces in Regions with Sublinear
Boundary}\label{sec:sub}

The main goal of this section is to prove the existence of proper hyperbolic minimal surfaces with non-empty boundary in $\r^3$ contained in the region above a negative sublinear graph.

Throughout this section $\mathcal{ N}$ will be the complex plane $\c.$

\begin{theorem} \label{th:sub}
Let $C$ denote the set $[-1,1] \times (0,1] \subset \r^2\equiv \c$ endowed with the conformal structure induced by $\c.$ 

Then there exists  $X \in \mathcal{ M}(C)$ satisfying that:
\begin{enumerate}[{\rm (1)}]
  \item $(x_1,x_3)\circ X:C \to \r^2$ is proper.
  \item If we set $f:C \to (-\infty,0],$ $f:=\min \{\frac{x_3\circ X}{|x_1\circ X|+1},0\},$ then $f=0$ on $(x_1\circ X)^{-1}((-\infty,0]),$ and $\lim_{n \to \infty} f(P_n)=0$ for any divergent sequence $\{P_n\}_{n \in \n}$ in $C.$
\end{enumerate}
\end{theorem}

\begin{proof} Let $D_n$ denote the rectangle $[-2,2] \times
[\frac{1}{n+1},2]\subset \r^2\equiv \c,$ $n \in \n.$ Label also
$D=[-2,2] \times [0,2].$ 

The immersion $X$ will be constructed recursively. Let us show the following

\begin{lemma} \label{lem:cu}
Fix $\epsilon_1\in(0,1).$ There exists a sequence of non-flat $X_n\in \mathcal{ M}(D),$ $k \in \n,$ such
that
\begin{enumerate}[{\rm (i)}]
\item $\|X_n-X_{n-1}\|_{1,D_{n-1}}<\epsilon_n,$ where
\[
\epsilon_n= \frac1{2^n}\min\left\{\epsilon_1\;,\;\min\left\{\big\|\frac{\partial X_k}{dz}\big\|_{0,D_k}\;|\; k=1,\ldots,n-1\right\}\right\}>0,
\]
for all $n \geq 2.$
\item $X_n([-2,2] \times \{\frac{1}{n+1}\}) \subset \Pi_{n}(\frac{1}{n}).$
\item $X_n(D_n-D_{n-1}) \subset \Pi_{n-1}(\frac{1}{n-1})\cup \Pi_n(\frac1{n})$  for all $n\geq 2.$
\item If $P\in D_n$ and  $(x_1\circ X_n)(P)<0,$ then $(x_3\circ X_n)(P)>1-\sum_{k=2}^n\epsilon_k>0.$ 
\end{enumerate}
\end{lemma}
\begin{proof}
Let us construct the sequence inductively. Take any non-flat $X_1\in
\mathcal{ M}(D)$ satisfying that $X_1(D_1) \subset \Pi_1(1).$ Notice that $X_1$ fulfills (ii) and (iv) whereas (i) and (iii) make no sense for $n=1.$ Assume there exists a non-flat immersion $X_{n-1},$ $n\geq 2,$
satisfying (i), (ii), (iii) and (iv), and let us construct
$X_{n}.$

Denote by $L:\r^3\to\r^3$ the rotation of angle $\frac1{n-1}$ around
the straight line parallel to the $x_2$-axis and containing the
point $(0,0,n-1).$ Notice that 
\begin{equation}\label{eq:giro2}
\text{$L(\Pi_{n-1}(\frac1{n-1}))=\Pi_{n-1}(0)\;$ and $\;L(\Pi_n(\frac1{n}))=\Pi_\zeta(-\frac{1}{n(n-1)}),$}
\end{equation}
 where
$\zeta=n-1+\cos(1/n) {\rm sec}(1/(n^2-n))$.

Call $Y=(Y_j)_{j=1,2,3}:=L\circ X_{n-1}\in \mathcal{ M}(D).$  From
(ii) and \eqref{eq:giro2}, we have
\begin{equation}\label{eq:x3oY}
Y([-2,2] \times \{{1}/{n}\})\subset \Pi_{n-1}(0).
\end{equation}

By continuity and equation (\ref{eq:x3oY}), there exists $\mu\in
(\frac1{n+1},\frac1{n}),$ close enough to $1/n$ so that
\begin{equation}\label{eq:x3oY''}
Y(\Theta)\subset \Pi_{n-1}(0),\quad \text{where $\Theta:=[-2,2]\times
[\mu,1/{n}],$}
\end{equation}
that is to say, $Y_3>n-1$ on $\Theta.$

Denote by $\Delta:=[-2,2]\times [0,\mu].$ Notice that
$D_n-D_{n-1}\subset \Delta\cup\Theta,$  and
$\emptyset=D_{n-1}\cap\Delta=D_{n-1}^\circ\cap\Theta^\circ
=\Theta^\circ\cap \Delta^\circ$ (see Figure \ref{fig:cuadrado}).

\begin{figure}[ht]
    \begin{center}
    \scalebox{0.40}{\includegraphics{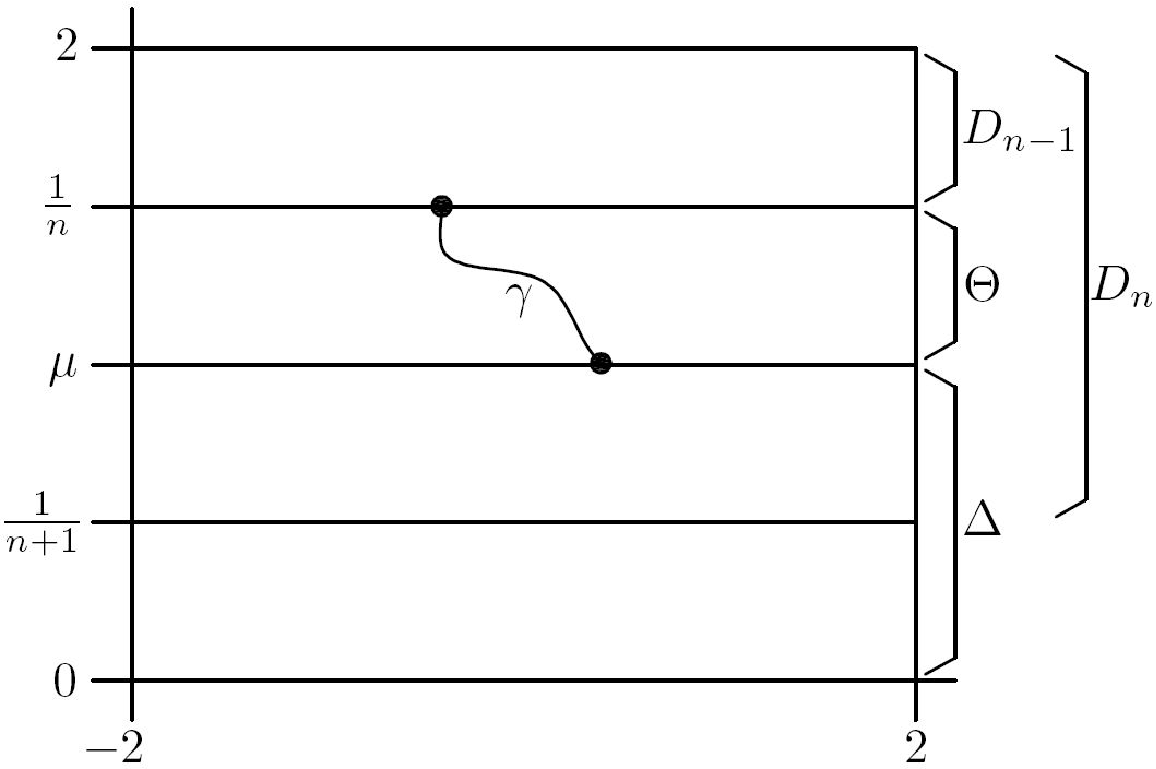}}

        \end{center}
        \vspace{-0.5cm}
\caption{The sets in $D$}\label{fig:cuadrado}
\end{figure}

Since $-\frac1{n(n-1)}\in (-\pi/2,0)$ and $Y(\Delta)$ is compact, then there exists $\lambda>0$ such that
\begin{equation}\label{eq:lam}
(-\lambda,0,0)+Y(\Delta) \subset \Pi_{\zeta}\big(-\frac{1}{n(n-1)}\big).
\end{equation}

The key idea to construct $X_n$ is similar to the one in (the second step of) the proof of Lemma \ref{lem:car0}. We deform $Y$ by pushing $(x_1,x_3)\circ Y (\Delta)\subset\r^2$ to the left in the direction of the $x_1$-axis a distance $\lambda,$ while preserving $Y_3$ on $D$  and hardly modifying $Y$ on $D_{n-1}.$ In this way we obtain a new immersion $Z\in \Mcal (D)$ such that $x_3\circ Z=Y_3$ on $D$ and $Z(\Delta)\subset L(\Pi_n(\frac1{n})).$   By \eqref{eq:x3oY''} and \eqref{eq:lam},  $X_n:= L^{-1} \circ Z$  will satisfy the desired properties. No matter the values of both $x_2\circ Z$ on $\Theta\cup\Delta$ and $x_1\circ Z$ on $\Theta.$

Consider $\gamma$ an analytic Jordan arc on $\Theta$ with
endpoints $Q_1\in  \partial(D_{n-1})$ and $Q_2\in \partial(\Delta)$
and otherwise disjoint from $\partial (\Theta),$ and meeting transversally $D_{n-1}$ and $\Delta$ (see Figure \ref{fig:cuadrado}). Moreover, we choose $\gamma$ so that $\partial Y_3$ never vanishes on $\gamma.$ Denote by
$\Lambda$ the admissible subset $\Lambda:=D_{n-1}\cup\gamma\cup\Delta$ in $\c$ and consider
$F_\varpi\in\mathcal{M}_\ggot^*(\Lambda),$ where $F=(F_j)_{j=1,2,3},$ satisfying
\begin{enumerate}[(A)]
\item $F=Y$ on $D_{n-1},$
\item $F_1=Y_1-\lambda$ on $\Delta,$
\item $F_3=Y_3$ and $(\partial F_\varpi)_3=\partial Y_3$ on $\Lambda.$
\end{enumerate}
%
%
The existence of $F_\varpi$ follows by similar arguments to those used in Claim \ref{ass:inter}.

Let $W\subset  \c$ be an open topological disc containing $D,$ and without loss of generality, suppose that $\partial Y_3$ extends holomorphically to $W.$
We can apply Theorem \ref{co:immaprox} to the data $W,$
$S=\Lambda,$  $F_\varpi$ and a $\xi\in(0,\epsilon_n)$ to obtain 
$Z=(Z_k)_{k=1,2,3}\in
\mathcal{M}(D)$ such that 
$\|Z- F_\varpi\|_{1,\Lambda}<\xi$  and $Z_3=F_3=Y_3.$ Then,
\begin{itemize}
\item $Z(\Theta)\subset \Pi_{n-1}(0)$ (take into account \eqref{eq:x3oY''} and that $Z_3=Y_3$),
\end{itemize}
and, if $\xi$ is chosen small enough, 

\begin{itemize}
\item $\|Z-Y\|_{1,D_{n-1}}<\epsilon_n.$
\item $Z(\Delta)\subset \Pi_\zeta(-\frac{1}{n(n-1)}).$ Use \eqref{eq:lam} and (B).
\item If $P\in D_{n-1}$ and $Z(P)\in L(\{x_1<0\}),$ then $Z(P)\in L(\{x_3>1-\sum_{k=2}^{n-1}\epsilon_k\}).$ Use (iv) and the induction hypothesis.
\end{itemize}

Define $X_n:=L^{-1}\circ Z\in\mathcal{ M}(D).$ From \eqref{eq:giro2} and translating the above
properties, we get
\begin{enumerate}[(a)]
\item $\|X_n-X_{n-1}\|_{1,D_{n-1}}<\epsilon_n.$
\item $X_n(\Theta)\subset \Pi_{n-1}(\frac{1}{n-1}).$
\item $X_n(\Delta)\subset \Pi_n(\frac{1}{n}).$
\item If $P\in D_{n-1}$ and $(x_1\circ X_n)(P)<0,$ then $(x_3\circ X_n)(P)>1-\sum_{k=2}^{n-1}\epsilon_k.$
\end{enumerate}

Property (a) directly gives (i). Since $[-2,2] \times
\{\frac{1}{n+1}\}\subset\Delta,$ (c) implies (ii). Taking into
account that $D_n-D_{n-1}\subset \Theta\cup\Delta,$ (iii) follows
from (b) and (c). Finally, (a) and (d) (respectively, (b) and (c)) give (iv) for points $P\in D_{n-1}$ (respectively, $P\in \Theta\cup\Delta$). 
\end{proof}

From (i) and Harnack's theorem, the sequence $\{X_n\}_{n\in\n}$
uniformly converges on compact sets of $(-2,2)
\times (0,2)$ to a
conformal minimal (possibly branched) immersion $\hat{X}:(-2,2)
\times (0,2)\to\r^3.$ From (i) and reasoning as in the proof of Theorem \ref{th:fun}, we deduce that $\hat{X}$ is an immersion and $X:=\hat{X}|_C\in\mathcal{
M}(C).$ 

Let us check that $X$ satisfies item (1). Denote by $C_n=[-1,1] \times [\frac{1}{n+1},1]\subset\c,$
$n\in\n.$ From (iii) we get that $\|(x_1,x_3)\circ X_n\|_{0,C_n-C_{n-1}^\circ}\geq \dist_{\r^3}(0,\Pi_{n-1}(\frac1{n-1})\cup \Pi_n(\frac1{n})).$ Then (i) gives $\|(x_1,x_3)\circ X\|_{0,C_n-C_{n-1}^\circ}\geq  \dist_{\r^3}(0,\Pi_{n-1}(\frac1{n-1})\cup \Pi_n(\frac1{n}))-\epsilon_1.$ Since $\lim_{n \to \infty} \dist_{\r^3}(0,\Pi_{n-1}(\frac1{n-1})\cup \Pi_n(\frac1{n}))=\infty,$ we infer that $(x_1,x_3)\circ X:C\to\r^2$ is proper.

Finally, let us show that $X$ satisfies item (2).
Consider $P\in C$ such that $(x_1\circ X)(P) <0.$ For $n$ large
enough, $P\in C_n\subset D_n$ and $(x_1\circ X_n)(P)<0$ as well.
Therefore (iv) gives $(x_3\circ X)(P)=\lim_{n\to\infty}(x_3\circ X_n)(P)\geq 1-\epsilon_1>0,$ and so 
$f(P)=0.$ Finally, consider  a divergent
sequence $\{P_n\}_{n\in\n}$ in $C$ with $(x_1\circ X)(P_n)\geq 0.$ For any $n\in\n$ we label $k(n)\in\n$ as the
natural number such that $P_n\in C_{k(n)}-C_{k(n)-1}$ and note that $\{k(n)\}_{n\in\n}$ is divergent. From (i), (iii) and the fact $(x_1\circ X)(P_n)\geq 0,$ one has
$(x_3+\tan(\frac1{k(n)-1})x_1)(X(P_n))> k(n)-2\epsilon_1.$
Hence, for $n$ large enough,
\begin{eqnarray*}
0\geq f(X(P_n))& \geq & \min\left\{
\frac{k(n)-2\epsilon}{x_1(X(P_n))+1}-\tan
\left(\frac1{k(n)-1}\right)\frac{x_1}{x_1+1}(X(P_n))\,,\,0\right\}\\
& \geq & -\tan
\left(\frac1{k(n)-1}\right),
\end{eqnarray*}
which converges to $0$ as $n$ goes to $\infty.$ This shows (2) and concludes the proof.
\end{proof}

By Caratheodory's Theorem, the set $C$ in the above Theorem is biholomorphic to the half disc $\overline{\d}_+,$ which corresponds to the statement of Theorem II in the introduction.


\begin{thebibliography}{999999}


\bibitem[\bf{AS}]{ahlfors} L.V. Ahlfors and L. Sario, {\em Riemann surfaces.} Princeton Univ. Press, Princeton, New Jersey, 1960.

\bibitem[\bf{Al}]{alar} A. Alarc\'{o}n, {\em On the existence of a proper conformal maximal disk in $\mathbb{L}^3$.}
 Differential Geom. Appl. {\bf 26}  (2008),  no. 2, 151-168.
 

\bibitem[\bf{AFL1}]{AFL} A. Alarc\'{o}n, I. Fern\'{a}ndez and F.J. L\'{o}pez, {\em Complete minimal surfaces and harmonic functions.}
Comment. Math. Helv. (In press).

\bibitem[\bf{AFL2}]{AFL2} A. Alarc\'{o}n, I. Fern\'{a}ndez and F.J. L\'{o}pez, {\em  Harmonic mappings and conformal minimal immersions of Riemann surfaces into $\r^N$.} Preprint (arXiv:1007.3124).
 


\bibitem[\bf{AL}]{AL} A. Alarc\'{o}n and F.J. L\'{o}pez, {\em Null curves in $\c^3$ and Calabi-Yau conjectures.} Preprint (arXiv:0912.2847).

\bibitem[\bf{Bi}]{bis} E. Bishop, {\em Mappings of partially analytic spaces.} Amer. J. Math. {\bf 83} (1961), 209-242.

\bibitem[\bf{CM}]{c-m1} T.H. Colding and W.P. Minicozzi II. {\em The Calabi-Yau conjectures for embedded surfaces}. Ann. of Math. (2)
{\bf 167} (2008), 211-243.

\bibitem[\bf{CKMR}]{c-k-m-r} P. Collin, R. Kusner, W.H. Meeks  III and H. Rosenberg, {\em The topology, geometry and conformal
structures of properly embedded minimal surfaces}  J. Differential
Geom. {\bf  67} (2004), 377-393.

\bibitem[\bf{DF}]{DF} B. Drinovec Drnov$\breve{{\rm s}}$ek and F. Forstneri$\breve{{\rm c}}$, {\em Holomorphic curves
in complex spaces.}  Duke Math. J. {\bf 139}  (2007), 203-253.

\bibitem[\bf{FK}]{farkas} H.M. Farkas and I. Kra, {\em Riemann surfaces.} Graduate Texts in Math. {\bf 71}, Springer Verlag, New York-Berlin, 1980.

\bibitem[\bf{FMM}]{f-m-m} L. Ferrer, F. Mart\'{i}n and W.H. Meeks III, {\em Existence of proper minimal surfaces of arbitrary topological type}. Preprint (arXiv:0903.4194).

\bibitem[\bf{HM}]{h-m} D. Hoffman and W.H. Meeks III, {\em The strong halfspace theorem for minimal surfaces.}  Invent. Math. {\bf 101}  (1990),
373-377.

\bibitem[\bf{JM}]{jorge-meeks}  L.P.M. Jorge and W.H. Meeks III, {\em The topology of complete minimal surfaces of finite total
 Gaussian curvature.} Topology {\bf 22} (1983), 203-221.

\bibitem[\bf{JX}]{jorge-xavier} L.P.M. Jorge and F. Xavier, {\em A complete minimal surface in $\r^3$ between two parallel planes.}
Ann. of Math. (2) {\bf 112} (1980), 203-206.

\bibitem[\bf{Lo1}]{lop3} F.J. L\'{o}pez, {\em Some Picard theorems for minimal surfaces.}  Trans. Amer. Math. Soc. {\bf 356} (2004), 703-733.

\bibitem[\bf{Lo2}]{lop1} F.J. L\'{o}pez, {\em Uniform Approximation by algebraic minimal surfaces in $\r^3.$} Preprint (arXiv:0903.3209).

\bibitem[\bf{LMM}]{lop-mar-mo} F.J. L\'{o}pez, F. Mart\'{i}n and S. Morales, {\em Adding handles to Nadirashvili's surfaces.}
J. Differential Geom. {\bf 60} (2002), 155-175.

\bibitem[\bf{LP}]{l-p} F.J. L\'{o}pez and J. P\'{e}rez, {\em Parabolicity and Gauss map of minimal surfaces.}
Indiana Univ. Math. J. {\bf  52} (2003),  1017-1026.

\bibitem[\bf{MP1}]{mp1} W.H. Meeks III and J. P\'{e}rez, {\em Conformal properties in classical minimal surface theory.}  Surveys in differential geometry. Vol. IX,  275--335, Surv. Differ. Geom., IX, Int. Press, Somerville, MA, 2004.

\bibitem[\bf{MP2}]{m-p} W.H. Meeks III and J. P\'{e}rez, {\em Finite type annular ends for harmonic functions.} Preprint (arXiv:0909.1963).

\bibitem[\bf{MPR}]{m-p-r} W.H. Meeks III, J. P\'{e}rez and A. Ros, {\em The embedded Calabi-Yau conjectures for finite genus}. In preparation.


\bibitem[\bf{Mo}]{mo} S. Morales, {\em On the existence of a proper minimal surface in $\r^3$ with the conformal type of a disk.}
Geom. Funct. Anal. {\bf 13} (2003), 1281-1301.

\bibitem[\bf{Nad}]{nadi} N. Nadirashvili, {\em Hadamard's and Calabi-Yau's conjectures on negatively curved and minimal surfaces.}
Invent. Math. {\bf 126} (1996), 457-465.

\bibitem[\bf{Nar}]{nar} R. Narasimhan, {\em Imbedding of holomorphically complete complex spaces.} Amer. J. Math. {\bf 82} (1960), 917-934.

\bibitem[\bf{Ne}]{neel} R.W. Neel, {\em Brownian motion and the parabolicity of minimal graphs.} Preprint (arXiv:0810.0669).

\bibitem[\bf{Os}]{osserman} R. Osserman, {\em A survey of minimal surfaces}. Dover Publications, New York, second edition, 1986.

\bibitem[\bf{Pe}]{P} J. P\'{e}rez, {\em Parabolicity and minimal surfaces.} Clay Math. Proc. {\bf 2},  Global theory of minimal surfaces, 163-174, Amer. Math. Soc., Providence, RI, 2005.

\bibitem[{\bf Pi}]{pirola} G.P. Pirola, {\em Algebraic curves and non-rigid minimal surfaces in the Euclidean
space.} Pacific J. Math. {\bf 183} (1998), 333-357.

\bibitem[\bf{Re}]{rem} R. Remmert, {\em Sur les espaces analytiques holomorphiquement s\'{e}parables et holomorphiquement convexes.} C. R. Acad. Sci. Paris {\bf 243} (1956), 118-121.

\bibitem[\bf{Ro}]{roy} H.L. Royden, {\em Function Theory on Compact Riemann Surfaces.} J. Analyse Math. {\bf 18} (1967), 295-327.

\bibitem[\bf{Sc1}]{sche1} S. Scheinberg, {\em Uniform approximation by functions analytic on a  Riemann surface.} Ann. of Math. (2) {\bf 108} (1978), 257-298.

\bibitem[\bf{Sc2}]{sche2}  S. Scheinberg, {\em Uniform approximation by meromorphic functions having prescribed poles.} Math. Ann. {\bf 243} (1979), 83-93.

\bibitem[\bf{SY}]{s-y} R. Schoen and S.T. Yau, {\em Lectures on Harmonic Maps.}
Conference Proceedings and Lecture Notes in Geometry and Topology,
II. International Press, Cambridge, MA, 1997.

\bibitem[\bf{UY}]{uy} M. Umehara and K. Yamada, {\em Maximal surfaces with singularities in Minkowski space.} Hokkaido Math. J. {\bf 35} (2006), 13-40.

\end{thebibliography}
\end{document}